\documentclass{amsart}
\usepackage{graphicx}
\usepackage{amssymb}
\usepackage{wrapfig}
\newtheorem{thm}{Theorem}[section]
\newtheorem{cor}[thm]{Corollary}
\newtheorem{lem}[thm]{Lemma}
\newtheorem{prop}[thm]{Proposition}
\theoremstyle{definition}
\newtheorem{defn}[thm]{Definition}
\theoremstyle{remark}
\newtheorem{rem}[thm]{Remark}
\numberwithin{equation}{section}
\newcommand{\norm}[1]{\left\Vert#1\right\Vert}
\newcommand{\abs}[1]{\left\vert#1\right\vert}
\newcommand{\set}[1]{\left\{#1\right\}}

\newcommand{\eps}{\varepsilon}

\newcommand{\dbar}{\bar\partial}
\newcommand{\ddbar}{\partial\bar\partial}

\DeclareMathOperator{\dom}{Dom}

\DeclareMathOperator{\re}{Re}
\DeclareMathOperator{\im}{Im}
\DeclareMathOperator{\supp}{supp}
\DeclareMathOperator{\dist}{dist}

\begin{document}

\title[The Diederich-Fornaess Index and Good Vector Fields]{The Diederich-Fornaess Index and Good Vector Fields}%
\author{Phillip S. Harrington}%
\address{SCEN 309, 1 University of Arkansas, Fayetteville, AR 72701}%
\email{psharrin@uark.edu}%


\begin{abstract}
  We consider the relationship between two sufficient conditions for regularity of the Bergman Projection on smooth, bounded, pseudoconvex domains.  We show that if the set of infinite type points is reasonably well-behaved, then the existence of a family of good vector fields in the sense of Boas and Straube implies that the Diederich-Fornaess Index of the domain is equal to $1$.
\end{abstract}
\maketitle

\tableofcontents


\section{Introduction}

Let $\Omega\subset\mathbb{C}^n$, and let $P$ denote the Bergman Projection, i.e., the orthogonal projection from $L^2(\Omega)$ to the space of $L^2$ holomorphic functions on $\Omega$.  One of the central questions in several complex variables is the following: if $f$ is smooth on $\overline\Omega$, is $Pf$ also smooth on $\overline\Omega$?  When the answer is affirmative for all $f\in C^\infty(\overline\Omega)$, we say that $P$ is globally regular on $\Omega$.  The Diederich-Fornaess worm domain \cite{DiFo77a} is a known counterexample, as shown by Christ \cite{Chr96} using work of Barrett \cite{Bar92}.  Our goal in the present paper is to examine the relationships between several known sufficient conditions for global regularity.

One of the most important sufficient conditions for global regularity is compactness of the $\dbar$-Neumann operator, but this is known to be strictly stronger than global regularity, so we will say very little about compactness in the present paper, except to refer the interested reader to Chapter 4 in \cite{Str10}.

The most straightforward sufficient condition for global regularity is the existence of a smooth defining function for $\Omega$ which is plurisubharmonic on $\partial\Omega$ \cite{BoSt91}.  This can be generalized in two different directions.  The first generalization is already found in \cite{BoSt91}, and involves the existence of a family of vector fields on $\partial\Omega$ that have good commutation properties with $\dbar$ (see Definition \ref{defn:good_vector_field} below for the precise definition).  This condition was explored and expanded by Boas and Straube in several subsequent works; we have adopted the version stated in \cite{BoSt99}, although we will make significant use of the equivalent condition developed in \cite{BoSt93}.

The second generalization was introduced by Kohn in \cite{Koh99}, building on a structure introduced by Diederich and Fornaess in \cite{DiFo77b}.  For a smooth, bounded, pseudoconvex domain in $\Omega$, the Diederich-Fornaess Index is defined to be the supremum over all exponents $0<\eta<1$ admitting a smooth defining function $\rho_\eta$ for $\Omega$ such that $-(-\rho_\eta)^\eta$ is plurisubharmonic on $\Omega$.  Kohn showed that global regularity is obtained when the Diederich-Fornaess Index is equal to $1$ and
\[
  \liminf_{\eta\rightarrow 1^-}\sqrt[3]{1-\eta}\sup_{\partial\Omega}\abs{\nabla \varphi_\eta}=0,
\]
where $\rho_\eta=e^{-\varphi_\eta}\delta$ and $\delta$ is the signed distance function for $\Omega$.  This technical condition has been refined in \cite{Har11} to
\begin{equation}
\label{eq:uniformity_condition}
  \liminf_{\eta\rightarrow 1^-}\sqrt{1-\eta}\sup_{\partial\Omega}\abs{\nabla \varphi_\eta}=0
\end{equation}
(see also \cite{PiZa14} for an alternative condition).  Herbig and Fornaess have shown that the Diederich-Fornaess Index is equal to $1$ whenever $\Omega$ has a defining function that is plurisubharmonic on the boundary in \cite{HeFo07} and \cite{HeFo08} (their construction also implies \eqref{eq:uniformity_condition}; see Remark 6.3 in \cite{Har11}), so this is a true generalization of Boas and Straube's original condition.

Our goal in the present paper is to show that when the set of infinite type points is reasonably well-behaved, then the good vector field condition of Boas and Straube implies that the Diederich-Fornaess Index is equal to $1$.  Although the precise meaning of ``well-behaved" will require some work to define (see Section \ref{sec:main_result} below), in this introduction we will discuss a few corollaries that are easier to state.

Our main theorem will show that when the set of infinite type points is foliated by complex submanifolds in a nice way, then the good vector field condition implies that the Diederich-Fornaess Index is equal to $1$.  We will refer to the submanifolds in this foliation as admissible leaves when they satisfy the requirements of our main theorem (see Definitions \ref{defn:leaf_0} and \ref{defn:leaf_m} below).  Our motivating example for this foliation is the following:
\begin{cor}
\label{cor:cross_product}
  Let $\Omega\subset\mathbb{C}^n$ be a smooth, bounded, pseudoconvex domain admitting a family of good vector fields.  Suppose that the set of infinite type points $K$ is contained in a neighborhood $U$ admitting a holomorphic map $g:U\rightarrow\mathbb{C}^{n-m}$ such that for every $p\in K$, $\partial\Omega\cap g^{-1}[g(p)]$ is a complex submanifold of dimension $m$ with smooth boundary and $g_j[K]\subset\mathbb{C}$ is a set with 2-dimensional Hausdorff measure zero for every $1\leq j\leq n-m$.  Then the Diederich-Fornaess Index of $\Omega$ is equal to $1$.
\end{cor}
For example, suppose $K$ is biholomorphic to $D\times C$, where $D$ is the unit disc in $\mathbb{C}$ and $C$ is the Cantor set.  Corollary \ref{cor:cross_product} will follow immediately from our main theorem by a result of Boas \cite{Boa88}, as we will show in Section \ref{sec:examples}.

However, we also wish to consider complex manifolds with non-smooth boundaries.  As discussed in \cite{BoSt93}, the obstruction to global regularity on the worm domain is the behavior of a certain cohomology class on the analytic annulus contained in the boundary.  One of our goals in this paper is to better understand the set of infinite type points when they are homotopic to the annulus but with simply connected interior, as illustrated by the two examples in Figure \ref{fig:examples} below.  We will obtain a relatively complete understanding of Example 1, but our methods will be inadequate for the study of Example 2, so we will conclude the paper with some open questions about this case.

\begin{figure}
  \centering
  \includegraphics[width=0.8\textwidth]{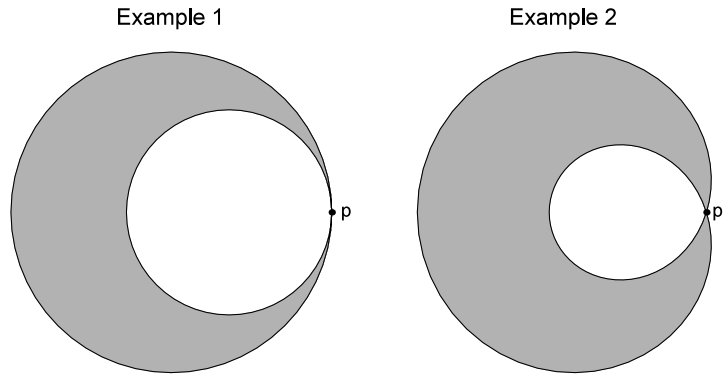}\\
  \caption{Examples of sets of infinite type points.  Example 1 is bounded by two circles intersecting only at $p$, but Example 2 satisfies an interior cone condition}\label{fig:examples}
\end{figure}

We will show that
\begin{prop}
\label{prop:annulus}
  Let $\Omega\subset\mathbb{C}^n$ be a smooth, bounded, pseudoconvex domain such that after a biholomorphic change of coordinates in a neighborhood of the set of infinite type points $K$, we have
  \[
    K=\{z\in\mathbb{C}^n:z_1\in S\text{ and }z_2=\cdots=z_n=0\},
  \]
  where $S\subset\mathbb{C}$ is the closure of a domain with smooth boundary except at $z=1$, $S\backslash\{1\}$ is simply connected, and for some $0<\gamma<1$ and $r>0$, $S\cap\overline{B(1,r)}$ is contained in the set $\{z\in\mathbb{C}:\abs{\im z}\geq\abs{\re z-1}^\gamma\}$, i.e., $S\cap\overline{B(1,r)}$ is contained between two algebraic curves with vertical tangent lines at the point $1$.  Then $\Omega$ admits a family of good vector fields.
\end{prop}
Note that Proposition \ref{prop:annulus} includes Example 1 in Figure \ref{fig:examples} (for any $\gamma>\frac{1}{2}$), but not Example 2.  Combining this with our main theorem, we will find that
\begin{cor}
\label{cor:annulus}
  Let $\Omega\subset\mathbb{C}^n$ be a smooth, bounded, pseudoconvex domain satisfying the hypotheses of Proposition \ref{prop:annulus}.  Then the Diederich-Fornaess Index of $\Omega$ is equal to $1$.
\end{cor}
In Section \ref{sec:examples}, we will construct a family of examples modeled on the Diederich-Fornaess worm domain \cite{DiFo77a} satisfying the hypotheses of Proposition \ref{prop:annulus} for any $0<\gamma<1$.  Because of their relationship to the worm domain, these examples will not admit a plurisubharmonic defining function, and even the existence of a family of good vector fields will require the full force of Definition \ref{defn:good_vector_field} (see Remark \ref{rem:strong_good_vector_field} below for details).

Unfortunately, our current assumptions are inadequate to prove \eqref{eq:uniformity_condition} as well.  This could be remedied by including an additional assumption on the growth rate of $\abs{\nabla\lambda_{M,r}}$ in Definitions \ref{defn:leaf_0} and \ref{defn:leaf_m} below, but we would lose results like Corollary \ref{cor:Property_P}, below.  Recall that a domain satisfies Catlin's Property $(P)$ if and only if for every $M>0$ there exists a smooth plurisubharmonic function $\lambda$ on $\overline\Omega$ such that $0\leq\lambda\leq 1$ on $\Omega$ and $i\ddbar\lambda\geq iM\ddbar\abs{z}^2$ on $\partial\Omega$.
\begin{cor}
\label{cor:Property_P}
  Let $\Omega\subset\mathbb{C}^n$ be a smooth, bounded pseudoconvex domain satisfying Catlin's Property $(P)$.  Then the Diederich-Fornaess Index of $\Omega$ is equal to $1$.
\end{cor}
\begin{rem}
  This Corollary appears to be known already, although a proof has not yet appeared in print.  We note it as a consequence of our main theorem, although a simpler and more direct proof is possible, as outlined in the opening paragraph of Section 5 in \cite{KLP16}.
\end{rem}
Since Catlin's Property $(P)$ requires no growth condition on the gradients of the weight functions involved, we run the risk of excluding examples if we impose a growth condition on our own weight functions simply to satisfy \eqref{eq:uniformity_condition}.

As a final observation, we note that sufficient conditions for global regularity on $\Omega$ are often related to sufficient conditions for the existence of a Stein neighborhood base for $\overline\Omega$.  We conclude this paper by noting that whenever our conditions imply that the Diederich-Fornaess Index is equal to $1$, it is also true that for every $0<\eta<1$ there exists a smooth defining function $\rho$ for $\Omega$ such that $\rho^{1/\eta}$ is plurisubharmonic outside of $\Omega$ (see the related results in \cite{HeFo07} and \cite{HeFo08}).

The outline for the remainder of the paper is as follows: in Section \ref{sec:main_result}, we will carefully define the structures that we need to provide a precise statement of our main result.  In Section \ref{sec:de_rham_cohomology}, we adapt ideas of Boas and Straube \cite{BoSt93} to show that the existence of a family of good vector fields implies that a certain $1$-form is exact (see \cite{Liu17b} for a similar analysis along these lines).  In Section \ref{sec:weight} we combine this with a family of weight functions with large hessians in order to build a weight function that is well-suited to the study of the Diederich-Fornaess Index (see \cite{Liu17a} for a more careful analysis of such weight functions).  We finally prove our main result in Section \ref{sec:global_construction}.  We show that these ideas can be easily adapted to constructing Stein neighborhood bases in Section \ref{sec:Stein}.  We conclude with a discussion of the corollaries to our main theorem in Section \ref{sec:examples}, including explicit examples and some open questions.

As in \cite{Liu17a}, we note that our results hold for any domain with $C^3$ boundary, although we have focused on smooth boundaries because this is necessary for the applications to global regularity.

The author would like to thank the referee for many corrections and helpful suggestions that have significantly improved the present paper.

\section{Statement of Main Result}
\label{sec:main_result}

We begin by inductively defining the sets of points that are amenable to our methods.  Our basic structure is an admissible leaf, which is essentially a complex submanifold that is maximal with respect to the set of infinite type points (see \cite{DAn93} for a precise definition and detailed exposition of finite and infinite type points).  However, an admissible leaf need not contain all of its own boundary points; those points which are excluded will need to be admissible leaves of a lower dimension.  See Remark \ref{rem:example_1} below to see how Figure \ref{fig:examples} illustrates this relationship.  To start our inductive definition, we consider the individual points that will be admissible.  For consistency of terminology in our later results, we will refer to a point as a leaf of dimension $0$.
\begin{defn}
\label{defn:leaf_0}
  For a compact set $K\subset\mathbb{C}^n$, let $p\in K$.  We say that $\{p\}$ is an \textbf{admissible leaf} of $K$ of dimension $0$ if for every $M>0$ and $r>0$ there exists a radius $0<R_{M,r}<r$ and a smooth function $\lambda_{M,r}$ defined on a neighborhood $U_{M,r}$ of $K\cap\overline{B(p,R_{M,r})}$ such that $0\leq\lambda_{M,r}\leq 1$ and $i\ddbar\lambda_{M,r}\geq i\frac{M}{R_{M,r}^2}\ddbar\abs{z}^2$ on $U_{M,r}$.
\end{defn}
Now that our base case has been established, we consider higher dimensional objects that will also be admissible.
\begin{defn}
\label{defn:leaf_m}
  For a compact set $K\subset\mathbb{C}^n$, let $V\subset K$ and $1\leq m\leq n-1$.  We say that $V$ is an \textbf{admissible leaf} of $K$ of dimension $m$ if there exists a neighborhood $U$ of $V$ and a holomorphic map $f:U\rightarrow\mathbb{C}^{n-m}$ such that
  \begin{enumerate}
    \item \label{item:leaf_m_vanishing} $f(z)=0$ for all $z\in V$,

    \item \label{item:leaf_m_lin_ind} $\{\partial f_1(z),\ldots,\partial f_{n-m}(z)\}$ is a linearly independent set over $\mathbb{C}$ for every $z\in V$,

    \item \label{item:leaf_m_property_p} for every $M>0$ and $r>0$ there exists a radius $0<R_{M,r}<r$ and a smooth function $\lambda_{M,r}$ defined on a neighborhood $U_{M,r}\subset\mathbb{C}^{n-m}$ of $f[K\cap U]\cap\overline{B(0,R_{M,r})}$ such that $0\leq\lambda_{M,r}\leq 1$ and $i\ddbar\lambda_{M,r}\geq i\frac{M}{R_{M,r}^2}\ddbar\abs{w}^2$ for $w\in U_{M,r}$,

    \item \label{item:leaf_m_topology} if $\tilde{V}\subset U$ is the vanishing set of $f$ and $V^\circ$ is the interior of $V$ relative to $\tilde{V}$, then $K\cap\tilde{V}=\overline{V^\circ}$,

    \item \label{item:leaf_m_paths} for every $p\in V$ there exists a neighborhood $U_p$ of $p$ and a constant $C_p>0$ such that any two points $z,w\in U_p\cap V^\circ$ are connected by a smooth curve in $U_p\cap V^\circ$ of length at most $C_p|z-w|$,

    \item \label{item:leaf_m_boundary} there are finitely many connected components of $\bar V\backslash V$, and each is the closure of an admissible leaf of $K$ of dimension $m'$ for some $0\leq m'<m$.
  \end{enumerate}
\end{defn}

\begin{rem}
  Hypotheses \eqref{item:leaf_m_vanishing} and \eqref{item:leaf_m_lin_ind} guarantee that $V$ is contained in an analytic variety $\tilde{V}$ that is nonsingular on $V$.  Since $V$ is not necessarily closed, this leaves open the possibility that $\tilde{V}$ might be singular at some boundary point of $V$.  However, these singular points must themselves be contained in an admissible leaf, as required by hypothesis \eqref{item:leaf_m_boundary}.
\end{rem}

\begin{rem}
  Hypothesis \eqref{item:leaf_m_property_p} guarantees that $\tilde V$ is maximal in $K$, in the following sense: if $f[K\cap U]$ is large enough to contain an analytic disc through the origin, then this disc will obstruct the existence of the weight functions $\lambda_{M,r}$.  On the other hand, if $f[K\cap U]$ has $2$-dimensional Hausdorff measure equal to zero, then hypothesis \eqref{item:leaf_m_property_p} is satisfied by a result of Boas \cite{Boa88}, which we have made use of in Corollary \ref{cor:cross_product}.  Motivated by these observations, we would like to replace hypothesis \eqref{item:leaf_m_property_p} with something like the requirement that the $2$-dimensional Hausdorff measure of $f[K\cap U]\cap\overline{B(0,r)}$ should vanish to order $o(r^2)$, as is the case in Corollary \ref{cor:annulus}.  Unfortunately, we have not found a precise measure-theoretic replacement for hypothesis \eqref{item:leaf_m_property_p}.
\end{rem}

\begin{rem}
  Hypothesis \eqref{item:leaf_m_topology} guarantees that $\tilde{V}$ is minimal with respect to $V$, since the interior of $V$ relative to $\tilde{V}$ must be nontrivial.
\end{rem}

\begin{rem}
  Hypothesis \eqref{item:leaf_m_paths} also guarantees that $V$ is locally connected in a strong sense, although we note that this is satisfied whenever the boundary of $V$ relative to $\tilde{V}$ is locally the graph of a continuous function with respect to some coordinate patch.  On the other hand, consider the domain in $\mathbb{C}$ parameterized by $V=\{e^{-t+i(\theta+t)}:t\geq 0,0\leq\theta\leq\pi\}$, i.e., near the origin $V$ is bounded by two logarithmic spirals.  The boundary of $V$ is not the graph of a continuous function for any coordinate patch containing the origin, but for any fixed $b>a>0$ and $0<\theta<\pi$, the length of the path $\gamma(s)=e^{-s+i(\theta+s)}$ for $a\leq s\leq b$ is equal to $\sqrt{2}(e^{-a}-e^{-b})$.  Hence, any two points $z,w\in V^\circ$ are connected by a path of length at most $\sqrt{2}||z|-|w||+|\arg(z/w)|\min\{|z|,|w|\}$, so hypothesis \eqref{item:leaf_m_paths} is still satisfied.
\end{rem}

\begin{rem}
  Hypothesis \eqref{item:leaf_m_boundary} requires that those boundary points of $V$ which are not in $V$ (and hence do not satisfy the local connectedness property of hypothesis \eqref{item:leaf_m_paths}) must themselves be contained in admissible leaves.
\end{rem}

\begin{rem}
\label{rem:example_1}
  Consider Example 1 in Figure \ref{fig:examples}.  $K$ itself is not an admissible leaf of dimension $1$, because hypothesis \eqref{item:leaf_m_paths} will fail at $p$ (the interior of $K$ is disconnected in a neighborhood of $p$).  However, we will show in Section \ref{sec:examples} that the point $p$ is an admissible leaf of dimension $0$ and hence $K\backslash\set{p}$ is an admissible leaf of dimension $1$ (see the proof of Corollary \ref{cor:annulus}).
\end{rem}

\begin{rem}
\label{rem:example_2}
  Example 2 in Figure \ref{fig:examples} is homeomorphic to the example on the left, but it is no longer admissible.  Suppose that $p$ is an admissible leaf of dimension $0$, and for $M>0$ and $r>0$ let $R_{M,r}$, $U_{M,r}$, and $\lambda_{M,r}$ be given by Definition \ref{defn:leaf_0}.  Let $\Gamma$ be any cone in $K\cap B(p,1)$ with vertex $p$.  Then $\lambda_M=\lambda_{M,r}(p+R_{M,r}(z-p))$ satisfies $i\ddbar\lambda_M\geq iM\ddbar\abs{z}^2$ and $0\leq\lambda_M\leq 1$ on $\Gamma$.  However, $\Gamma$ contains a disc, so we have a contradiction when $M$ is large.
\end{rem}

Recall that $\delta$ denotes the signed distance function, i.e., $\delta(z)=\dist(z,\partial\Omega)$ for $z\notin\Omega$ and $\delta(z)=-\dist(z,\partial\Omega)$ for $z\in\Omega$.  The following definition can be found in \cite{BoSt99} (see the hypotheses of Theorem 13), with some additional references and motivation.
\begin{defn}
\label{defn:good_vector_field}
  Let $\Omega\subset\mathbb{C}^n$ be a bounded pseudoconvex domain with smooth boundary.  We say that $\Omega$ admits a \textbf{family of good vector fields} if there exists a positive constant $C>0$ such that for every $\eps>0$ there exists a $(1,0)$ vector field $X_\eps$ with smooth coefficients on some neighborhood $U_\eps$ of the set of infinite type points $K\subset\partial\Omega$ satisfying
  \begin{enumerate}
    \item $C^{-1}<\abs{X_\eps\delta}<C$ on $U_\eps$,

    \item $\abs{\arg X_\eps\delta}<\eps$ on $U_\eps$, and

    \item $\abs{\partial\delta([X_\eps,\partial/\partial\bar z_j])}<\eps$ on $U_\eps$ for all $1\leq j\leq n$.
  \end{enumerate}
\end{defn}

With these definitions in mind, we are able to state our main result:
\begin{thm}
\label{thm:Main}
  Let $\Omega\subset\mathbb{C}^n$ be a smooth, bounded, pseudoconvex domain admitting a family of good vector fields.  Suppose that every point of infinite type belongs to an admissible leaf of the set of infinite type points.  Then
  \begin{enumerate}
    \item the Diederich-Fornaess index of $\Omega$ is equal to $1$ and
    \item for every $0<\eta<1$ there exists a smooth defining function $\rho$ and a neigh\-borhood $U$ of $\overline\Omega$ such that $\rho^{1/\eta}$ is strictly plurisubharmonic on $\Omega^c\cap U$.
  \end{enumerate}
\end{thm}

\section{De Rham Cohomology of Admissible Leaves}
\label{sec:de_rham_cohomology}

Let $\Omega\subset\mathbb{C}^n$ be a smooth, bounded, pseudoconvex domain, and let $\delta$ denote the signed distance function for this domain.  For any smooth submanifold $E$ of $\mathbb{C}^n$ (real or complex), we let $T_p(E)$ denote the real tangent space of $E$ at $p$.  We let $T^{1,0}_p(E)\subset\mathbb{C}T_p(E)$ denote the subbundle of the complexified tangent space $\mathbb{C}T_p(E)$ consisting of type $(1,0)$ vector fields, with $T^{0,1}_p(E)=\overline{T^{1,0}_p(E)}$.  When $E=\partial\Omega$, we let $\mathcal{N}_p(\partial\Omega)\subset T_p(\partial\Omega)$ denote the null space of the Levi-form at $p$, i.e., the set of vectors $v\in\mathbb{C}^n$ such that $\sum_{j=1}^n v^j\frac{\partial\delta}{\partial z_j}(p)=0$ and $\sum_{j,k=1}^n v^j\frac{\partial^2\delta}{\partial z_j\partial\bar z_k}(p)\bar\tau^k=0$ whenever $\tau\in T_p(\partial\Omega)$.  Note that this is equivalent to
\begin{equation}
\label{eq:nullspace_Levi_form}
  \sum_{j=1}^n v^j\frac{\partial^2\delta}{\partial z_j\partial\bar z_k}(p)=4\sum_{j,\ell=1}^n v^j\frac{\partial^2\delta}{\partial z_j\partial\bar z_\ell}(p)\frac{\partial\delta}{\partial z_\ell}(p)\frac{\partial\delta}{\partial\bar z_k}(p)
\end{equation}
for every $1\leq k\leq n$.

In \cite{BoSt93}, Boas and Straube show that the existence of a family of good vector fields is closely tied to the properties of D'Angelo's $1$-form $\alpha$.  In local coordinates, this form can be written
\[
  \alpha=4\sum_{j=1}^n\re\left(\frac{\partial\delta}{\partial \bar z_j}\dbar\left(\frac{\partial\delta}{\partial z_j}\right)\right),
\]
as derived from (5.85) in \cite{Str10} (using $|\partial\delta|^2=\frac{1}{2}$ on $\partial\Omega$).  We will frequently use the notation $\pi_{1,0}\alpha$ (resp.\ $\pi_{0,1}\alpha$) to denote the projection of $\alpha$ onto its $(1,0)$-component (resp.\ $(0,1)$-component).  Since $|\nabla\delta|\equiv 1$ in a neighborhood of $\partial\Omega$, we have $0=\partial\sum_{j=1}^n\abs{\frac{\partial\delta}{\partial z_j}}^2$, so
\begin{equation}
\label{eq:alpha_projection}
  \pi_{1,0}\alpha=2\sum_{j=1}^n\frac{\partial\delta}{\partial z_j}\partial\left(\frac{\partial\delta}{\partial\bar z_j}\right)=-2\sum_{j=1}^n\frac{\partial\delta}{\partial\bar z_j}\partial\left(\frac{\partial\delta}{\partial z_j}\right).
\end{equation}
We note that to have intrinsic meaning, $\alpha$ should only be defined on the boundary of $\Omega$.  However, we will find it useful to have a smooth extension of $\alpha$ to a neighborhood of $\partial\Omega$, and these formulas in local coordinates will suffice to give us a canonical extension to any neighborhood on which $\delta$ is smooth.  Such a neighborhood can always be found on bounded domains with smooth boundaries by a result of Krantz and Parks \cite{KrPa81}

We will also find it helpful to use the real, positive semi-definite, $(1,1)$-form $\beta$ defined by
\[
  \beta=i\sum_{j,k=1}^n \partial\left(\frac{\partial\delta}{\partial z_j}\right)\left(I_{jk}-4\frac{\partial\delta}{\partial\bar z_j}\frac{\partial\delta}{\partial z_k}\right)\wedge\dbar\left(\frac{\partial\delta}{\partial \bar z_k}\right),
\]
where $I_{jk}$ is the identity matrix.  As usual, we define $d^c=i(\partial-\dbar)$.  Then
\[
  d^c\alpha=4\sum_{j=1}^n\re\left(d^c\left(\frac{\partial\delta}{\partial \bar z_j}\right)\wedge\dbar\left(\frac{\partial\delta}{\partial z_j}\right)\right)+4\sum_{j=1}^n\re\left(\frac{\partial\delta}{\partial \bar z_j}i\ddbar\left(\frac{\partial\delta}{\partial z_j}\right)\right).
\]
Since
\begin{multline*}
  0=i\ddbar\sum_{j=1}^n\abs{\frac{\partial\delta}{\partial z_j}}^2=\sum_{j=1}^n 2\re\left(\frac{\partial\delta}{\partial \bar z_j}i\ddbar\left(\frac{\partial\delta}{\partial z_j}\right)\right)\\
  +\sum_{j=1}^n i\partial\left(\frac{\partial\delta}{\partial \bar z_j}\right)\wedge \dbar\left(\frac{\partial\delta}{\partial z_j}\right)+\sum_{j=1}^n i\partial\left(\frac{\partial\delta}{\partial z_j}\right)\wedge \dbar\left(\frac{\partial\delta}{\partial\bar z_j}\right),
\end{multline*}
we have
\begin{multline*}
  d^c\alpha=-4\sum_{j=1}^n\re\left(i\dbar\left(\frac{\partial\delta}{\partial \bar z_j}\right)\wedge\dbar\left(\frac{\partial\delta}{\partial z_j}\right)\right)\\
  +2\sum_{j=1}^n i\partial\left(\frac{\partial\delta}{\partial \bar z_j}\right)\wedge\dbar\left(\frac{\partial\delta}{\partial z_j}\right)-2\sum_{j=1}^n i\partial\left(\frac{\partial\delta}{\partial z_j}\right)\wedge \dbar\left(\frac{\partial\delta}{\partial\bar z_j}\right).
\end{multline*}
If we restrict to the null space $\mathcal{N}_p(\partial\Omega)$ of the Levi-form at $p$, then we may use \eqref{eq:nullspace_Levi_form} and \eqref{eq:alpha_projection} to show
\[
  \partial\left(\frac{\partial\delta}{\partial\bar z_j}\right)\bigg|_{\mathcal{N}_p(\partial\Omega)}=2\frac{\partial\delta}{\partial\bar z_j}\pi_{1,0}\alpha\bigg|_{\mathcal{N}_p(\partial\Omega)}.
\]
This gives us
\begin{multline*}
  d^c\alpha|_{\mathcal{N}_p(\partial\Omega)\wedge\overline{\mathcal{N}_p(\partial\Omega)}}=-8\sum_{j=1}^n\re\left(i\frac{\partial\delta}{\partial z_j}\dbar\left(\frac{\partial\delta}{\partial \bar z_j}\right)\wedge\pi_{0,1}\alpha\right)\\
  +2i\pi_{1,0}\alpha\wedge\pi_{0,1}\alpha-2\sum_{j=1}^n i\partial\left(\frac{\partial\delta}{\partial z_j}\right)\wedge \dbar\left(\frac{\partial\delta}{\partial\bar z_j}\right).
\end{multline*}
The first term vanishes by another application of \eqref{eq:alpha_projection}, so
\begin{equation}
\label{eq:alpha_derivative}
  d^c\alpha|_{\mathcal{N}_p(\partial\Omega)\wedge\overline{\mathcal{N}_p(\partial\Omega)}}=-2\beta.
\end{equation}

Our main result in this section is the following:
\begin{lem}
\label{lem:h_V_construction}
  Let $\Omega\subset\mathbb{C}^n$ be a smooth, bounded, pseudoconvex domain admitting a family of good vector fields.  If $K\subset\partial\Omega$ is the set of infinite type points, let $V\subset K$ be an admissible leaf of $K$ of dimension $m$ for some $1\leq m\leq n-1$, and let $\tilde{V}$ denote the minimal analytic variety containing $V$.  For every $\eps>0$, if we define
  \[
    V_\eps:=\{z\in V:\dist(z,\bar V\backslash V)\geq\eps\},
  \]
  then there exists a smooth real-valued function $h_{V,\eps}$ on $\mathbb{C}^n$ such that
  \begin{align}
    \label{eq:h_first_derivative} d h_{V,\eps}|_{T_p(\tilde{V})}&=\alpha\text{ for any }p\in V_\eps,\\
    \label{eq:h_second_derivative} i\ddbar h_{V,\eps}|_{T^{1,0}_p(\tilde{V})\wedge T^{0,1}_p(\tilde{V})}&=-\beta\text{ for any }p\in V_\eps,
  \end{align}
  and $\abs{h_{V,\eps}}\leq\frac{1}{2}\log C$ on $V_\eps$ for a constant $C>1$ independent of $V$.
\end{lem}

\begin{proof}

Let $\{X_\eps\}$ be a family of good vector fields.  We may assume $0<\eps<\pi$.  Since $\abs{\arg X_\eps\delta}<\eps$, we can use the principal branch of the logarithm to define $h_\eps=\frac{1}{2}\log\left(X_\eps\delta\right)$.  By assumption, $\abs{\re h_\eps}<\frac{1}{2}\log C$ and $\abs{\im h_\eps}<\frac{1}{2}\eps$ on $U_\eps$.

We compute
\[
  \dbar h_\eps=\frac{1}{2}(X_\eps\delta)^{-1}([\dbar,X_\eps]\delta+X_\eps\dbar\delta)=-\frac{1}{2}e^{-2h_\eps}\left(\sum_{j=1}^n\partial\delta\left(\left[X_\eps,\frac{\partial}{\partial\bar z_j}\right]\right)d\bar z_j-X_\eps\dbar\delta\right).
\]
Set $X_\eps^\tau=X_\eps-4e^{2h_\eps}\sum_{j=1}^n\frac{\partial\delta}{\partial \bar z_j}\frac{\partial}{\partial z_j}$, so that $X_\eps^\tau$ represents the tangential component of $X^\eps$.  This gives us
\[
  \dbar h_\eps-\pi_{0,1}\alpha=-\frac{1}{2}e^{-2h_\eps}\left(\sum_{j=1}^n\partial\delta\left(\left[X_\eps,\frac{\partial}{\partial\bar z_j}\right]\right)d\bar z_j-X_\eps^\tau\dbar\delta\right).
\]
By the definition of the null space of the Levi-form, $X_\eps^\tau\dbar\delta|_{\overline{\mathcal{N}_p(\partial\Omega)}}=0$ for any $p\in K$.  Hence,
\begin{equation}
\label{eq:dbar_h_eps}
  \left(\dbar h_\eps-\pi_{0,1}\alpha\right)\Bigg|_{\overline{\mathcal{N}_p(\partial\Omega)}}=
  -\frac{1}{2}e^{-2h_\eps}\sum_{j=1}^n\partial\delta\left(\left[X_\eps,\frac{\partial}{\partial\bar z_j}\right]\right)d\bar z_j.
\end{equation}

If $C(K)$ is the space of continuous functions on $K$ equipped with the $L^\infty$ norm, then $\{h_\eps|_K\}$ is uniformly bounded in $C(K)$.  By Alaoglu's Theorem, there exists $\tilde h\in ((C(K))^*)^*$ that is a weak$^*$ limit of some subsequence $\{h_{\eps_k}\}$.  Set $h_k=h_{\eps_k}$, so we have $\left<T,h_k\right>\rightarrow \left<\tilde h,T\right>$ for every $T\in(C(K))^*$.  Given $p\in K$, note that $\left<T(p),h_k\right>:=h_k(p)$ is a bounded linear functional on $C(K)$, so we may identify $\tilde h$ with a real-valued function on $K$ defined by $h(p):=\left<\tilde h,T(p)\right>=\lim_k h_k(p)$.  Since the right-hand side of \eqref{eq:dbar_h_eps} approaches zero by assumption, we have
\[
  \lim_{k\rightarrow\infty}\left(\bar X h_k\right)(p)=\alpha(\bar X)
\]
for each $p\in K$ and $X\in\mathcal{N}_p(\partial\Omega)$.

We say that $X\in T^{1,0}_p(\partial\Omega)$ is tangent to an analytic disc at $p$ if there exists an open neighborhood $\mathcal{O}$ of the origin in $\mathbb{C}$ and a holomorphic map $\varphi:\mathcal{O}\rightarrow\partial\Omega$ such that $\varphi(0)=p$ and $X=\left(\varphi_*\frac{\partial}{\partial z}\right)|_{z=0}$.  Recall the generalized Cauchy Integral Formula (see Theorem 2.1.1 in \cite{ChSh01}): for any bounded domain $D\subset\mathbb{C}$ with $C^1$ boundary, $u\in C^1(\overline{D})$, and $z\in D$, we have
\begin{equation}
\label{eq:cauchy_integral}
  u(z)=\frac{1}{2\pi i}\left(\int_{\partial D}\frac{u(\zeta)}{\zeta-z}d\zeta+\iint_D\frac{\partial u/\partial\bar\zeta}{\zeta-z}d\zeta\wedge d\bar\zeta\right).
\end{equation}
Suppose $X\in T^{1,0}_p(\partial\Omega)$ is tangent to an analytic disc at $p$, and let $\mathcal{O}\subset\mathbb{C}$ and $\varphi:\mathcal{O}\rightarrow\partial\Omega$ define the analytic disc in $\partial\Omega$ to which $X$ is tangential.  Since $\varphi^*(\bar X h_k)=\frac{\partial}{\partial\bar z}\varphi^* h_k$, we may use \eqref{eq:cauchy_integral} with the dominated convergence theorem to find that
\[
  \varphi^*h(z)=\frac{1}{2\pi i}\int_{\partial D}\frac{\varphi^*h(\zeta)}{\zeta-z}d\zeta
  +\frac{1}{2\pi i}\iint_D\frac{1}{\zeta-z}\varphi^*\left(\alpha(\bar X)\right)(\zeta)d\zeta\wedge d\bar\zeta
\]
for any open set $D$ containing $z$ that is relatively compact in $\mathcal{O}$.  Since $\int_{\partial D}\frac{\varphi^*h(\zeta)}{\zeta-z}d\zeta$ is holomorphic in $z$, we may use Theorem 2.1.2 in \cite{ChSh01} to show $\frac{\partial}{\partial\bar z}\varphi^* h(z)=\varphi^*\left(\alpha(\bar X)\right)$ for $z\in D$.  Since this hold for all such imbeddings of analytic discs, we conclude
\begin{equation}
\label{eq:h_derivatives}
  \bar X h(p)=\alpha(\bar X)
\end{equation}
for all $p\in K$ and $X\in T^{1,0}_p(\partial\Omega)$ tangent to an analytic disc at $p$.

Define $V^\circ$ as in Definition \ref{defn:leaf_m} \eqref{item:leaf_m_topology}.  Since $V^\circ$ is a complex submanifold, every vector tangent to $V^\circ$ is tangent to an analytic disc, so for any piecewise $C^1$ path $\gamma$ in $V^\circ$ connecting $\gamma(0)$ and $\gamma(1)$, we have shown
\begin{equation}
\label{eq:de_rham_integral}
  h(\gamma(1))-h(\gamma(0))=\int_\gamma \alpha.
\end{equation}
Fix $p\in V\backslash V^\circ$.  Since $p\in\overline{V^\circ}$ by Definition \ref{defn:leaf_m} \eqref{item:leaf_m_topology}, there must exist a sequence $\{z_j\}\subset V^\circ$ converging to $p$, and by restricting to a subsequence we may assume that $\{h(z_j)\}$ also converges.  Suppose that $\{w_j\}\subset V^\circ$ is a second such sequence.  By Definition \ref{defn:leaf_m} \eqref{item:leaf_m_paths}, there must exist a neighborhood $U_p$ of $p$ and a constant $C_p>0$ such that any two points $z,w\in U_p\cap V^\circ$ are connected by a smooth path in $U_p\cap V^\circ$ of length at most $C_p|z-w|$.  For $j$ sufficiently large, $z_j\in U_{p}$ and $w_j\in U_{p}$, so there exists a smooth curve $\gamma$ in $U_{p}\cap V^\circ$ connecting $z_j$ to $w_j$, but then \eqref{eq:de_rham_integral} gives us $\abs{h(z_j)-h(w_j)}\leq C_p|z_j-w_j|\sup_{\partial\Omega}|\alpha|$ for all $j$ sufficiently large.  We conclude that $\{h(z_j)\}$ and $\{h(w_j)\}$ must have the same limit, and hence we may define a unique continuous function $h_V$ on $V$ satisfying $h_V|_{V^\circ}=h$.

Let $f$ be given by Definition \ref{defn:leaf_m}.  We denote the singular set of $\tilde V$ by $\tilde V_{\mathrm{sing}}$ and define $\tilde V_{\mathrm{reg}}=\tilde{V}\backslash\tilde{V}_{\mathrm{sing}}$.  We will make frequent use of the fact that Definition \ref{defn:leaf_m} \eqref{item:leaf_m_lin_ind} guarantees that $V\subset\tilde{V}_{\mathrm{reg}}$.  For any $w\in\tilde{V}_{\mathrm{reg}}$, the implicit function theorem implies the existence of a $\mathbb{C}^{n-m}$-valued holomorphic map $a_w(z)$ in a neighborhood of $w$ satisfying $a_w(w)=0$ and
\[
  f_j\left(z_k-\sum_{\ell=1}^{n-m}a_w^\ell(z)\frac{\partial\bar f_\ell}{\partial\bar w_k}(w)\right)=0.
\]
On any compact subset of $\tilde{V}_{\mathrm{reg}}$, $a_w(z)$ is defined on a neighborhood of $w$ of uniform size.  With this in mind, the implicit function theorem also guarantees that $a_w(z)$ is conjugate holomorphic (and hence smooth) in $w$ on $\tilde{V}_{\mathrm{reg}}$.  Set
\[
  \pi_w(z)=z-\sum_{\ell=1}^{n-m}a_w^\ell(z)\frac{\partial\bar f_\ell}{\partial\bar w}(w),
\]
so that $\pi_w(z)$ is a holomorphic projection from a neighborhood of $w$ onto $\tilde V$.  Once again, $\pi_w(z)$ is conjugate holomorphic (and hence smooth) in $w$ on $\tilde{V}_{\mathrm{reg}}$.  Since $\alpha$ is closed when restricted to $T(\tilde V)\wedge T(\tilde V)$ on $V$, $\pi_w^*\alpha$ is closed whenever $\pi_w(z)\in V$.  On the other hand, we also have a family of $(1,0)$-vector fields $\{L^w_{1},\ldots,L^w_{n-m}\}$ defined in a neighborhood of $w$ such that $(\partial(\pi_{w})_k)(L^w_{\ell})=0$ and $\partial f_j(L^w_{\ell})=I_{j\ell}$ for all $1\leq k\leq n$ and $1\leq j,\ell\leq n-m$.  When $w\in V^\circ$, there exists a unique smooth function $P_w(z)$ defined in a neighborhood of $w$ satisfying $P_w(w)=0$ and
\[
  dP_w=\pi_w^*\alpha+d\left( 2\re\left(\sum_{\ell=1}^{n-m}\alpha(L^w_{\ell}) f_\ell\right)\right).
\]
Since $f=0$ on $V$ and $\set{L^w_{\ell}}$ is dual to $\set{\partial f_\ell}$, we can check that $dP_w$ is equal to $\alpha$ on a neighborhood of $w$ in $V$.

For $\sigma\geq 1$, let $P_\sigma(w,z)$ be the unique $\sigma$th degree Taylor polynomial for $P_w(z)$ centered at $w$.  For $\tilde\eps>0$, let $\tilde {V}_{\tilde\eps}$ denote the set of $w\in\tilde{V}$ such that $\dist(w,\tilde{V}_{\mathrm{sing}})\geq\tilde\eps$.  Note that the coefficients of $P_\sigma(w,z)$ can be explicitly computed in terms of derivatives of $\alpha$, $\pi_w$, $f$, and $L^w$ in a way that depends smoothly on $w$ on $\tilde V_{\tilde\eps}$, so $P_\sigma(w,z)$ can be extended smoothly to $w\in V\cap\tilde V_{\tilde\eps}$.  Furthermore, $P_\sigma(w,w)=0$ and
\begin{equation}
\label{eq:dP_minus_alpha}
  \abs{d_z P_\sigma(w,z)-\alpha|_z}\leq O(|z-w|^\sigma)
\end{equation}
for $z,w\in V\cap\tilde V_{\tilde\eps}$, with a constant depending on $\tilde\eps>0$.  For $p\in V$, let $U_p$ and $C_p$ be given by Definition \ref{defn:leaf_m} \eqref{item:leaf_m_paths}, so that for $z,w\in U_p\cap V^\circ$ there exists a smooth path in $U_p\cap V^\circ$ of length at most $C_p|z-w|$ connecting $z$ to $w$.  Choose $\tilde\eps>0$ sufficiently small so that $\dist(p,\tilde V_{\mathrm{sing}})>\tilde\eps$, and choose a neighborhood $\tilde U_p$ of $p$ that is sufficiently small so that for $z,w\in\tilde U_p\cap V^\circ$, there exists a smooth path in $U_p\cap V^\circ\cap\tilde V_{\tilde\eps}$ of length at most $C_p|z-w|$ connecting $z$ to $w$.  Then \eqref{eq:de_rham_integral} and \eqref{eq:dP_minus_alpha} imply that for $z,w\in \tilde U_p\cap V^\circ$ we have $|h(z)-h(w)-P_\sigma(w,z)|\leq O(|z-w|^{\sigma+1})$ with a constant depending only on $C_p$ and the constant in \eqref{eq:dP_minus_alpha}.  Since $P_\sigma(w,z)$ depends continuously on $w$ and $z$, we may use continuity to conclude $|h_V(z)-h_V(w)-P_\sigma(w,z)|\leq O(|z-w|^{\sigma+1})$ for $z,w\in\tilde U_p\cap V$.  Hence, $h_V$ is $C^\sigma$ on the compact set $V_\eps$ in the sense of Whitney \cite{Whi34} for any $\eps>0$.  We may use the Whitney extension theorem to extend $h_V$ from $V_\eps$ to a smooth real-valued function $h_{V,\eps}$ defined on all of $\mathbb{C}^n$.  We immediately obtain \eqref{eq:h_first_derivative}.  Since $T_p(\tilde{V})\subset\mathcal{N}_p(\partial\Omega)$, we obtain \eqref{eq:h_second_derivative} from \eqref{eq:alpha_derivative}.

\end{proof}

\section{Good Weight Functions}
\label{sec:weight}

Now that we may use Lemma \ref{lem:h_V_construction} to build a good function on each admissible leaf, we need good weight functions to handle directions that are transverse to each leaf.  We will proceed by induction.  Since the base case will have greater applicability, we set it aside with its own lemma.

\begin{lem}
\label{lem:weight_near_point}
  Let $K\subset\partial\Omega$ be a compact set and let $p\in K$ such that $\{p\}$ is an admissible leaf of $K$ of dimension $0$.  Then for every $0<\eta<1$, $r>0$, and $B>A>0$ there exists a neighborhood $U_p\subset B(p,r)$ of $p$ and a smooth, real-valued function $\varphi_p$ on $U_p$ such that $\varphi_p(p)\geq A$, $\varphi_p\leq -A$ on $\partial U_p\cap K$, $\abs{\varphi_p}\leq B$ on $U_p$, and
  \begin{equation}
  \label{eq:weight_near_point}
    i\ddbar\varphi_p+2\beta-i\frac{\eta}{1-\eta}(\partial\varphi_p-2\pi_{1,0}\alpha)\wedge(\dbar\varphi_p-2\pi_{0,1}\alpha)>0.
  \end{equation}
\end{lem}

\begin{rem}
\label{rem:finite_type}
  We emphasize that $K$ in this lemma is not necessarily the set of infinite type points.  In fact, our hypotheses are satisfied near any points of finite type with $K=\partial\Omega\cap\overline{B(p,R)}$ for some $R>0$ sufficiently small.  Let $\Omega_p\subset\Omega$ be a smooth, pseudoconvex domain such that $\partial\Omega_p\cap\partial\Omega=K$ and $\partial\Omega_p\cap\Omega$ is strictly pseudoconvex.  Then $\partial\Omega_p$ satisfies Catlin's Property $(P)$ (see Theorems 2 and 4 in \cite{Cat84}).  Hence, given $M>0$ and $r>0$, we may choose any $0<R_{M,r}<r$, and there will exist a plurisubharmonic function $\phi\in C^\infty(\overline\Omega_p)$ with $0\leq\phi\leq 1$ such that $i\ddbar\phi\geq i\frac{4M}{R_{M,r}^2}\ddbar\abs{z}^2$ on $\partial\Omega_p$.  We may smoothly extend $\phi$ to $\mathbb{C}^n$, in which case $\lambda_{M,r}=\frac{1}{2}\phi$ will satisfy Definition \ref{defn:leaf_0}.
\end{rem}

\begin{proof}
  Choose ${\zeta}>\frac{\eta}{1-\eta}$ and $M>\frac{(e^{A\zeta}-e^{-B\zeta})}{e^{-A\zeta}-e^{-B\zeta}}$.  Since nothing is lost by assuming that $r$ is smaller, we may assume that our given $r>0$ is sufficiently small so that
  \begin{multline}
  \label{eq:r_upper_bound}
    {\zeta}^{-1}e^{-B\zeta}(M(e^{-A\zeta}-e^{-B\zeta})-(e^{A\zeta}-e^{-B\zeta}))r^{-2}i\ddbar\abs{z}^2>\\
    i\frac{4\eta {\zeta}}{(1-\eta){\zeta}-\eta}\pi_{1,0}\alpha\wedge\pi_{0,1}\alpha-2\beta,
  \end{multline}
  on some neighborhood of $K$.  Let $R_{M,r}$, $U_{M,r}$, and $\lambda_{M,r}$ be given by Definition \ref{defn:leaf_0}.  Set $U_p=B(p,R_{M,r})\cap U_{M,r}$ and
  \[
    \varphi_p(z)=-{\zeta}^{-1}\log\left(e^{-A\zeta}+(e^{A\zeta}-e^{-B\zeta})R_{M,r}^{-2}|z-p|^2-(e^{-A\zeta}-e^{-B\zeta})\lambda_{M,r}(z)\right).
  \]
  Observe that $\varphi_p(p)\geq A$ but $\varphi_p\leq -A$ when $\abs{z-p}=R_{M,r}$.  On $U_p$, we have $-{\zeta}^{-1}\log(e^{-A\zeta}+e^{A\zeta}-e^{-B\zeta})\leq\varphi_p\leq B$, and since $e^{-A\zeta}+e^{A\zeta}<e^{-B\zeta}+e^{B\zeta}$ when $B>A>0$, we conclude $\varphi_p>-B$ as well.

  Now we compute
  \begin{multline*}
    i\ddbar(-e^{-{\zeta}\varphi_{p}})=-(e^{A\zeta}-e^{-B\zeta})R_{M,r}^{-2}i\ddbar\abs{z}^2+i(e^{-A\zeta}-e^{-B\zeta})\ddbar\lambda_{M,r}\\
    \geq(M(e^{-A\zeta}-e^{-B\zeta})-(e^{A\zeta}-e^{-B\zeta}))R_{M,r}^{-2}i\ddbar\abs{z}^2
  \end{multline*}
  on $U_{M,r}$.  Since
  \[
    i\ddbar(-e^{-{\zeta}\varphi_{p}})=i\zeta e^{-{\zeta}\varphi_{p}}\ddbar\varphi_{p}-i\zeta^2e^{-{\zeta}\varphi_{p}}\partial\varphi_{p}\wedge\dbar\varphi_{p},
  \]
  we have
  \[
    i\ddbar\varphi_{p}\geq {\zeta}^{-1}e^{{\zeta}\varphi_{p}}(M(e^{-A\zeta}-e^{-B\zeta})-(e^{A\zeta}-e^{-B\zeta}))R_{M,r}^{-2}i\ddbar\abs{z}^2+i\zeta\partial\varphi_{p}\wedge\dbar\varphi_{p}.
  \]
  Since
  \[
    i\left(\left(\zeta-\frac{\eta}{1-\eta}\right)\partial\varphi_{p}+\frac{2\eta}{1-\eta}\pi_{1,0}\alpha\right)\wedge\left(\left(\zeta-\frac{\eta}{1-\eta}\right)\dbar\varphi_{p}+\frac{2\eta}{1-\eta}\pi_{0,1}\alpha\right) \geq 0,
  \]
  we may rearrange terms using the identity $\left(\frac{\eta}{1-\eta}\right)^2=\left(\frac{\eta {\zeta}}{(1-\eta){\zeta}-\eta}-\frac{\eta}{1-\eta}\right)\left(\zeta-\frac{\eta}{1-\eta}\right)$ and divide by $\left(\zeta-\frac{\eta}{1-\eta}\right)$ to obtain
  \begin{multline*}
    i\zeta\partial\varphi_{p}\wedge\dbar\varphi_{p}\geq i\frac{\eta}{1-\eta}(\partial\varphi_{p}-2\pi_{1,0}\alpha)\wedge(\dbar\varphi_{p}-2\pi_{0,1}\alpha)\\
    -i\frac{4\eta {\zeta}}{(1-\eta){\zeta}-\eta}\pi_{1,0}\alpha\wedge\pi_{0,1}\alpha.
  \end{multline*}
  Since we already know $\varphi_p>-B$, \eqref{eq:weight_near_point} will follow provided that
  \begin{multline*}
    {\zeta}^{-1}e^{-B\zeta}(M(e^{-A\zeta}-e^{-B\zeta})-(e^{A\zeta}-e^{-B\zeta}))R_{M,r}^{-2}i\ddbar\abs{z}^2>\\
    i\frac{4\eta {\zeta}}{(1-\eta){\zeta}-\eta}\pi_{1,0}\alpha\wedge\pi_{0,1}\alpha-2\beta,
  \end{multline*}
  but this follows from \eqref{eq:r_upper_bound}, so the proof is complete.
\end{proof}

Before moving up to the higher dimensional case, we need a technical lemma.  The following lemma generalizes a standard technique for patching plurisubharmonic functions by taking a supremum and regularizing.

\begin{lem}
\label{lem:patching}
  For a finite set $J\subset\mathbb{N}$, let $\{\mathcal{O}_j\}_{j\in J}$ be a collection of open subsets of $\mathbb{C}^n$ such that $\alpha$ and $\beta$ are defined and smooth on $\overline{\mathcal{O}_j}$, and for some $0<\eta<1$ and every $j\in J$ let $\varphi_j\in C^\infty(\overline{\mathcal{O}_j})$ satisfy
  \begin{equation}
  \label{eq:weight_estimate}
    i\ddbar\varphi_j+2\beta-i\frac{\eta}{1-\eta}(\partial\varphi_j-2\pi_{1,0}\alpha)\wedge(\dbar\varphi_j-2\pi_{0,1}\alpha)>0
  \end{equation}
  on $\mathcal{O}_j$.  Suppose $\mathcal{O}_0\subset\mathbb{C}^n$ is an open set satisfying $\overline{\mathcal{O}_0}\subset\bigcup_{j\in J}\mathcal{O}_j$ and for every $j\in J$ and $z\in\partial\mathcal{O}_j\cap\overline{\mathcal{O}_0}$ there exists $k\in J$ such that $z\in\mathcal{O}_k$ and $\varphi_j(z)<\varphi_k(z)$.  Then for any $\epsilon>0$ there exists $\varphi_0\in C^\infty(\overline{\mathcal{O}_0})$ satisfying \eqref{eq:weight_estimate} on $\overline{\mathcal{O}_0}$ such that
  \[
    \epsilon>\varphi_0(z)-\max_{\{j\in J:z\in\mathcal{O}_j\}}\varphi_j(z)\geq 0
  \]
  on $\overline{\mathcal{O}_0}$ and $\varphi_0=\varphi_j$ on $\overline{\mathcal{O}_0}\cap\mathcal{O}_j\backslash\bigcup_{k\in J\backslash\{j\}}\mathcal{O}_k$.
\end{lem}

\begin{proof}
  Let $\psi\in C^\infty_0(\mathbb{R}^{|J|})$ be a nonnegative radially symmetric function satisfying $\supp\psi=\overline{B(0,1)}$ and $\int\psi=1$.  Define $\chi\in C^\infty(\mathbb{R}^{|J|})$ by
  \[
    \chi(x)=\int_{\mathbb{R}^{|J|}}\psi(x-y)\left(\max_{j\in J}y_j\right) dV_y.
  \]
  For any $j\in J$, $\frac{\partial}{\partial x_j}\psi(x-y)=-\frac{\partial}{\partial y_j}\psi(x-y)$, so integration by parts (used with caution, since $\max_{j\in J}y_j$ is continuous but only piecewise differentiable) will give us
  \[
    \frac{\partial}{\partial x_j}\chi(x)=\int_{\{y\in\mathbb{R}^{|J|}:y_j\geq y_k\text{ for all }k\in J\}}\psi(x-y)dV_y\geq 0.
  \]
  Hence, $\sum_{j\in J}\frac{\partial\chi}{\partial x_j}(x)\equiv 1$.  Note that $\chi$ is convex since $\max_{j\in J}y_j$ is convex with respect to $y\in\mathbb{R}^{|J|}$.  Furthermore, $\chi(x)=x_j$ whenever $j\in J$ and $x_j\geq x_k+1$ for all $k\in J\backslash\{j\}$ by the mean value property.  Combining this with convexity, we see that
  \[
    \chi(x)\geq\max_{j\in J}x_j.
  \]
  Since the Lipschitz constant of $\max_{j\in J}y_j$ is equal to $1$, we have
  \[
    \chi(x)-\max_{j\in J}x_j=\int_{\mathbb{R}^{|J|}}\psi(x-y)\left(\max_{j\in J}y_j-\max_{j\in J}x_j\right)dV_y\leq 1.
  \]

  As a technical convenience, we set $C=\min_{j\in J}\inf_{z\in\overline{\mathcal{O}_j}}\varphi_j(z)$, and extend each $\varphi_j$ to $\mathbb{C}^n$ by defining $\varphi_j(z)=C$ whenever $z\notin\overline{\mathcal{O}_j}$.  This notation and our hypotheses allow us to define $\max_{j\in J}\varphi_j(z)$ as a continuous function on $\overline{\mathcal{O}_0}$ without restricting $J$ to those indices such that $z\in\mathcal{O}_j$.  Since $\max_{j\in J}\varphi_j(z)$ is continuous on $\overline{\mathcal{O}_0}$, we may fix
  \[
    0<\xi<\min\set{\epsilon,\min_{j\in J}\inf_{z\in\partial\mathcal{O}_j\cap\overline{\mathcal{O}_0}}\max_{k\in J}(\varphi_k(z)-\varphi_j(z))}.
  \]
  We define
  \[
    \varphi_0(z)=\xi\cdot\chi\left((\xi^{-1}\varphi_j(z))_{j\in J}\right)
  \]
  whenever $z\in\overline{\mathcal{O}_0}\cap\mathcal{O}_j\cap\mathcal{O}_k$ for $j,k\in J$ with $j\neq k$, and $\varphi_0(z)=\varphi_j(z)$ when $z\in\overline{\mathcal{O}_0}\cap\mathcal{O}_j\backslash\bigcup_{k\in J\backslash\{j\}}\mathcal{O}_k$ for some $j\in J$.  When $\varphi_j(z)\geq\varphi_k(z)+\xi$ for some $j\in J$ and for any $k\in J\backslash\{j\}$, then $\xi\cdot\chi\left((\xi^{-1}\varphi_j(z))_{j\in J}\right)=\varphi_j(z)$, so $\varphi_0$ is smooth on $\overline{\mathcal{O}_0}$.  Our estimates for $\chi$ give us
  \[
    0\leq\varphi_0(z)-\max_{j\in J}\varphi_j(z)\leq\xi<\epsilon.
  \]

  Let $a_j(z)=\frac{\partial\chi}{\partial x_j}\left((\xi^{-1}\varphi_\ell(z))_{\ell\in J}\right)$, and note that we have already shown that $a_j(z)\geq 0$ for all $j\in J$ and $\sum_{j\in J}a_j(z)=1$.  Since $(\sqrt{a_j})_{j\in J}$ is a unit-length vector, the matrix $(I_{jk}-\sqrt{a_j}\sqrt{a_k})_{j,k\in J}$ is positive semi-definite on $\mathbb{R}^{|J|}$, where $I_{jk}$ denotes the identity matrix.  If we left-multiply and right-multiply this matrix by the diagonal matrix with diagonal entries $\sqrt{a_j}$, we see that the matrix $(a_j I_{jk}-a_j a_k)_{j,k\in J}$ is also positive semi-definite.  Given a family of $(1,0)$-forms $\{\theta_j\}_{j\in J}$, this implies the inequality
  \begin{equation}
  \label{eq:convexity_estimate}
    \sum_{j\in J}ia_j\theta_j\wedge\bar\theta_j\geq i\left(\sum_{j\in J}a_j\theta_j\right)\wedge\left(\sum_{k\in J}a_k\bar\theta_k\right).
  \end{equation}
  We compute
  \[
    \partial\varphi_0=\sum_{j\in J}a_j\partial\varphi_j.
  \]
  Let $b_{jk}(z)=\frac{\partial^2\chi}{\partial x_j\partial x_k}\left((\xi^{-1}\varphi_\ell(z))_{\ell\in J}\right)$.  Then $(b_{jk}(z))_{j,k\in J}$ is a positive semi-definite symmetric matrix.  We compute
  \[
    \ddbar\varphi_0=\sum_{j\in J}a_j\ddbar\varphi_j+\xi^{-1}\sum_{j,k\in J}b_{jk}\partial\varphi_j\wedge\dbar\varphi_k.
  \]
  Since $a_j\geq 0$ for all $j\in J$, $a_j>0$ for at least one $j\in J$, $(b_{jk}(z))_{j,k\in J}$ is positive semi-definite, and \eqref{eq:weight_estimate} holds for all $j\in J$, we have
  \[
    i\ddbar\varphi_0>\sum_{j\in J}a_j\left(-2\beta+i\frac{\eta}{1-\eta}(\partial\varphi_j-2\pi_{1,0}\alpha)\wedge(\dbar\varphi_j-2\pi_{0,1}\alpha)\right).
  \]
  Using $\sum_{j\in J}a_j=1$ and \eqref{eq:convexity_estimate} with $\theta_j=\partial\varphi_j-2\pi_{1,0}\alpha$, we see that $\varphi_0$ must also satisfy \eqref{eq:weight_estimate}.
\end{proof}

Now, we are ready to construct weight functions near admissible leaves of dimension $m>0$.

\begin{lem}
\label{lem:weight_near_leaf}
  Let $\Omega\subset\mathbb{C}^n$ be a smooth, bounded, pseudoconvex domain admitting a family of good vector fields.  If $K\subset\partial\Omega$ is the set of infinite type points, let $V\subset K$ be an admissible leaf of $K$ of dimension $m$ for some $0\leq m\leq n-1$, $0<\eta<1$, $A>0$, and $B>A+2m\log C$ where $C$ is the constant given in Lemma \ref{lem:h_V_construction}.  Then for any neighborhood $\mathcal{O}$ of $\bar V$ there exists a neighborhood $U_V\subset\mathcal{O}$ of $\bar V$ and a smooth, real-valued function $\varphi_V$ on $\overline{U_V}$ such that $\varphi_V\geq A$ on $\bar V$, $\varphi_V\leq -A$ on $\partial U_V\cap K$, and on $U_V$ we have $\abs{\varphi_V}\leq B$ and
  \begin{equation}
  \label{eq:weight_near_leaf}
    i\ddbar\varphi_V+2\beta-i\frac{\eta}{1-\eta}(\partial\varphi_V-2\pi_{1,0}\alpha)\wedge(\dbar\varphi_V-2\pi_{0,1}\alpha)>0.
  \end{equation}
\end{lem}

\begin{proof}
  When $m=0$, this follows immediately from Lemma \ref{lem:weight_near_point}.  We assume that $m>0$ and Lemma \ref{lem:weight_near_leaf} has already been proven for all $0\leq m'<m$.  Fix $\tilde{A}>0$ and $\tilde{B}>0$ such that $A+2\log C<\tilde{B}<\tilde{A}<B-2(m-1)\log C$.  Let $\{V_j\}_{j\in J}$ denote the set of all connected components of $\bar V\backslash V$, and apply Lemma \ref{lem:weight_near_leaf} to each component to obtain a set of neighborhoods $\{U_{j}\}$ and functions $\{\varphi_{j}\}$ such that $\abs{\varphi_{j}}<B$ on $\overline{U_{j}}$, $\varphi_{j}\geq\tilde{A}$ on $\bar V_j$, and $\varphi_{j}\leq-\tilde{A}$ on $\partial U_{j}\cap K$.  Since each $U_j$ can be chosen arbitrarily small, we may assume that the sets $\{\overline{U_j}\}_{j\in J}$ are disjoint.

  Choose $\eps>0$ sufficiently small so that $z\in U_{j}$ and $\varphi_{j}(z)>\frac{1}{2}(\tilde{A}+\tilde{B})$ whenever $\dist(z,V_j)<\eps$, and let $h_{\eps}$ be given by Lemma \ref{lem:h_V_construction}.  To minimize extraneous subscripts, we suppress the $V$ in the subscript for $h$.

  Let $f$ be given by Definition \ref{defn:leaf_m}.  Fix $D>0$ so that $|z|<D$ on $\overline{V}$.  Fix constants $E>0$, $\zeta>0$, and $M>0$ satisfying $\frac{\tilde{B}-A}{2\log C}>E>1$, ${\zeta}>\frac{\eta}{1-\eta}$, and $M>\frac{2C^{E\zeta}(e^{A\zeta}-e^{-\tilde{B}\zeta})}{C^{-E\zeta}e^{-A\zeta}-C^{E\zeta}e^{-\tilde{B}\zeta}}$.  For any $r>0$, set
  \begin{multline*}
    \Theta(r)=
    2i\ddbar h_{\eps}+2\beta-i\frac{4\eta\zeta}{(1-\eta)\zeta-\eta}(\partial h_{\eps}-\pi_{1,0}\alpha)\wedge(\dbar h_{\eps}-\pi_{0,1}\alpha)\\
    +\frac{1}{2}\zeta^{-1}e^{-\tilde{B}\zeta}(C^{-2E\zeta}e^{-A\zeta}-e^{-\tilde{B}\zeta})D^{-2}i\ddbar\abs{z}^2\\
    +\zeta^{-1}e^{-\tilde{B}\zeta}\left(\frac{1}{2}(C^{-2E\zeta}e^{-A\zeta}-e^{-\tilde{B}\zeta})M-(e^{A\zeta}-e^{-\tilde{B}\zeta})\right) r^{-2}i\ddbar|f(z)|^2.
  \end{multline*}
  If $\tilde{V}$ is the minimal analytic variety containing $V$ and $p\in V_\eps$, then $\partial f(\tau)=0$ for every $\tau\in T_p^{1,0}(\tilde{V})$.  By \eqref{eq:h_first_derivative} and \eqref{eq:h_second_derivative}, we have
  \[
    \Theta(r)|_{T_p^{1,0}(\tilde{V})\wedge T_p^{0,1}(\tilde{V})}=\frac{1}{2}\zeta^{-1}e^{-\tilde{B}\zeta}(C^{-2E\zeta}e^{-A\zeta}-e^{-\tilde{B}\zeta})D^{-2}i\ddbar\abs{z}^2,
  \]
  so $\Theta(r)$ is uniformly positive definite in tangential directions.  Similarly, if $\nu\in T^{1,0}_p(\mathbb{C}^n)$ is in the orthogonal complement of $T^{1,0}(\tilde{V})$, then $\abs{\Theta(r)(\tau\wedge\bar\nu)}$ is uniformly bounded by a constant that is independent of $r$.  Finally, $i\ddbar\abs{f(z)}^2(\nu\wedge\bar\nu)\geq iF\ddbar\abs{z}^2(\nu\wedge\bar\nu)$ for some constant $F>0$ that is independent of $p\in V_\eps$ and $\nu\in T^{1,0}_p(\mathbb{C}^n)$ that is orthogonal to $T^{1,0}_p(\tilde{V})$.  Note that $F$ may depend on $\eps$ if $\{\partial f_1,\ldots,\partial f_{n-m}\}$ is not a linearly independent set on $\bar V\backslash V$, but this dependence will not be relevant since we have already fixed $\eps$.  Hence, we may choose $r>0$ sufficiently small so that $\Theta(r)$ is positive definite on $V_\eps$.  Since $V_\eps$ is closed, there exists a neighborhood $U_\eps\subset B(0,D)$ of $V_\eps$ on which $\Theta(r)>0$.  We may further require that $\overline{U_\eps}\subset U$, where $U$ is given by Definition \eqref{defn:leaf_m}.  Since Lemma \ref{lem:h_V_construction} implies $\abs{h_{\eps}}\leq\frac{1}{2}\log C$, we may assume that $U_\eps$ has been chosen sufficiently small so that $\abs{h_{\eps}}<\frac{E}{2}\log C$ on $U_\eps$.

  Suppose $z\in K\cap\partial U_\eps$ satisfies $f(z)=0$ and $\varphi_{j}(z)\leq\frac{1}{2}(\tilde{A}+\tilde{B})$ if there exists $j\in J$ such that $z\in U_j$.  Then $z\in\overline{V}$ by Definition \ref{defn:leaf_m} \eqref{item:leaf_m_topology}, but $\dist(z,\bar V\backslash V)\geq\eps$ since either $\varphi_{j}(z)\leq\frac{1}{2}(\tilde{A}+\tilde{B})$ or $z\notin U_j$ for all $j\in J$, so $z\in V_\eps$.  Hence $z\in V_\eps\cap\partial U_\eps$, but this is impossible since $U_\eps$ is a neighborhood of $V_\eps$, so we have a contradiction.  Consequently,
  \[
    \inf\left\{|f(z)|:z\in K\cap\partial U_\eps\text{ such that }\varphi_{j}(z)\leq\frac{1}{2}(\tilde{A}+\tilde{B})\text{ or }z\notin U_j\text{ for all }j\in J\right\}
  \]
  is strictly positive.  Since $\Theta(r)>0$ remains true when we decrease the size of $r>0$, we may assume that $r>0$ has been chosen sufficiently small so that if $z\in K\cap\partial U_\eps$ and $\varphi_{j}(z)\leq\frac{1}{2}(\tilde{A}+\tilde{B})$ whenever $z\in U_j$, then $|f(z)|\geq r$.

  Now that we have established values for $M$ and $r$, let $R_{M,r}$, $\lambda_{M,r}$, and $U_{M,r}$ be given by Definition \ref{defn:leaf_m} \eqref{item:leaf_m_property_p}.  Let $\dom\phi$ denote the set of $z\in U_\eps$ for which $f(z)\in U_{M,r}$, and define
  \begin{multline*}
    \phi(z)=\frac{1}{2}(C^{-E\zeta}e^{-A\zeta}-C^{E\zeta}e^{-\tilde{B}\zeta})(D^{-2}|z|^2+(\lambda_{M,r}\circ f)(z))\\
    -C^{E\zeta}(e^{A\zeta}-e^{-\tilde{B}\zeta}) R_{M,r}^{-2}|f(z)|^2-C^{-E\zeta}e^{-A\zeta}.
  \end{multline*}
  Note that our bounds on $E$ and $\tilde B$ imply that $C^{-E\zeta}e^{-A\zeta}-C^{E\zeta}e^{-\tilde{B}\zeta}>0$ and $e^{A\zeta}-e^{-\tilde{B}\zeta}>0$.  For $z\in\dom\phi$, $z\in U_\eps$, so we have $0\leq |z|<D$.  Since $z\in\dom\phi$ implies that $f(z)\in U_{M,r}$, we know that $\lambda_{M,r}(f(z))$ is well-defined, so we also make use of the estimate $0\leq\lambda_{M,r}\circ f\leq 1$.  Whenever $z\in\dom\phi$ satisfies $|f(z)|\leq R_{M,r}$ we have
  \[
    -C^{E\zeta}(e^{A\zeta}-e^{-\tilde{B}\zeta})-C^{-E\zeta}e^{-A\zeta}\leq\phi(z)\leq -C^{E\zeta}e^{-\tilde{B}\zeta}.
  \]
  If $z\in V_\eps$, then we know that $z\in\dom\phi$ and $|f(z)|=0$, so we have $\phi(z)\geq-C^{-E\zeta}e^{-A\zeta}$.  When $z\in\dom\phi$ satisfies $|f(z)|>R_{M,r}$ we have $\phi(z)<-C^{E\zeta}e^{A\zeta}$, and if $|f(z)|$ is uniformly bounded above $R_{M,r}$ then $\phi(z)$ will be uniformly bounded below $-C^{E\zeta}e^{A\zeta}$.  On $\dom\phi$, we also define
  \[
    \varphi_{\eps}(z)=2h_{\eps}(z)-{\zeta}^{-1}\log(-\phi(z)).
  \]
  Using our estimates for $\phi$ and the fact that $|h_\eps|<\frac{E}{2}\log C$ on $U_\eps\subset\dom\phi$, we see that when $z\in\dom\phi$ and $|f(z)|\leq R_{M,r}$ we have
  \[
    -E\log C-{\zeta}^{-1}\log(C^{E\zeta}(e^{A\zeta}-e^{-\tilde{B}\zeta})+C^{-E\zeta}e^{-A\zeta})\leq\varphi_{\eps}(z)\leq \tilde{B},
  \]
  but then $\abs{\varphi_{\eps}(z)}\leq \tilde{B}$ since $C^{-E\zeta}e^{-A\zeta}+C^{E\zeta}e^{A\zeta}<C^{-E\zeta}e^{\tilde{B}\zeta}+C^{E\zeta}e^{-\tilde{B}\zeta}$ when $\tilde{B}-E\log C>A+E\log C$.  On $V_\eps$, we have $\varphi_{\eps}(z)\geq A$.  When $z\in\dom\phi$ and $|f(z)|>R_{M,r}$, we have $\varphi_{\eps}(z)<-A$, and if $|f(z)|$ is uniformly bounded above $R_{M,r}$ then $\varphi_\eps(z)$ will be uniformly bounded below $-A$.

  By assumption,
  \begin{multline*}
    i\ddbar\phi\geq\frac{1}{2}(C^{-E\zeta}e^{-A\zeta}-C^{E\zeta}e^{-\tilde{B}\zeta})(D^{-2}i\ddbar\abs{z}^2+R_{M,r}^{-2}iM\ddbar\abs{f(z)}^2)\\
    -C^{E\zeta}(e^{A\zeta}-e^{-\tilde{B}\zeta}) R_{M,r}^{-2}i\ddbar|f(z)|^2,
  \end{multline*}
  so our choice of $M$ guarantees that $\phi$ is strictly plurisubharmonic.  Since $\phi=-e^{\zeta(2h_{\eps}-\varphi_{\eps})}$ and
  \begin{multline*}
    i\ddbar(-e^{\zeta(2h_{\eps}-\varphi_{\eps})})=e^{\zeta(2h_{\eps}-\varphi_{\eps})}\zeta(i\ddbar\varphi_{\eps}-2i\ddbar h_{\eps})\\
    -e^{\zeta(2h_{\eps}-\varphi_{\eps})}i\zeta^2(\partial\varphi_{\eps}-2\partial h_{\eps})\wedge(\dbar\varphi_{\eps}-2\dbar h_{\eps}),
  \end{multline*}
  we have
  \begin{multline*}
    i\ddbar\varphi_{\eps}\geq 2i\ddbar h_{\eps}\\
    +\frac{1}{2}\zeta^{-1}e^{\zeta(\varphi_{\eps}-2h_{\eps})}(C^{-E\zeta}e^{-A\zeta}-C^{E\zeta}e^{-\tilde{B}\zeta})(D^{-2}i\ddbar\abs{z}^2+R_{M,r}^{-2}iM\ddbar\abs{f(z)}^2)\\
    -\zeta^{-1}e^{\zeta(\varphi_{\eps}-2h_{\eps})}C^{E\zeta}(e^{A\zeta}-e^{-\tilde{B}\zeta}) R_{M,r}^{-2}i\ddbar|f(z)|^2\\
    +i\zeta(\partial\varphi_{\eps}-2\partial h_{\eps})\wedge(\dbar\varphi_{\eps}-2\dbar h_{\eps}).
  \end{multline*}
  Since
  \begin{multline*}
    i\left(\left(\zeta-\frac{\eta}{1-\eta}\right)(\partial\varphi_{\eps}-2\partial h_{\eps})+\frac{2\eta}{1-\eta}(\pi_{1,0}\alpha-\partial h_\eps)\right)\\
    \wedge\left(\left(\zeta-\frac{\eta}{1-\eta}\right)(\dbar\varphi_{\eps}-2\dbar h_{\eps})+\frac{2\eta}{1-\eta}(\pi_{0,1}\alpha-\dbar h_\eps)\right) \geq 0,
  \end{multline*}
  we may rearrange terms using the identity $\left(\frac{\eta}{1-\eta}\right)^2=\left(\frac{\eta {\zeta}}{(1-\eta){\zeta}-\eta}-\frac{\eta}{1-\eta}\right)\left(\zeta-\frac{\eta}{1-\eta}\right)$ and divide by $\left(\zeta-\frac{\eta}{1-\eta}\right)$ to obtain
  \begin{multline*}
    i\zeta(\partial\varphi_{\eps}-2\partial h_{\eps})\wedge(\dbar\varphi_{\eps}-2\dbar h_{\eps})\geq i\frac{\eta}{1-\eta}(\partial\varphi_{\eps}-2\pi_{1,0}\alpha)\wedge(\dbar\varphi_{\eps}-2\pi_{0,1}\alpha)\\
    -i\frac{4\eta\zeta}{(1-\eta)\zeta-\eta}(\partial h_{\eps}-\pi_{1,0}\alpha)\wedge(\dbar h_{\eps}-\pi_{0,1}\alpha).
  \end{multline*}
  Hence,
  \begin{multline*}
    i\ddbar\varphi_{\eps}+2\beta-i\frac{\eta}{1-\eta}(\partial\varphi_{\eps}-2\pi_{1,0}\alpha)\wedge(\dbar\varphi_{\eps}-2\pi_{0,1}\alpha)\geq\\
    2i\ddbar h_{\eps}+2\beta-i\frac{4\eta\zeta}{(1-\eta)\zeta-\eta}(\partial h_{\eps}-\pi_{1,0}\alpha)\wedge(\dbar h_{\eps}-\pi_{0,1}\alpha)\\
    +\frac{1}{2}\zeta^{-1}e^{\zeta(\varphi_{\eps}-2h_{\eps})}(C^{-E\zeta}e^{-A\zeta}-C^{E\zeta}e^{-\tilde{B}\zeta})(D^{-2}i\ddbar\abs{z}^2+R_{M,r}^{-2}iM\ddbar\abs{f(z)}^2)\\
    -\zeta^{-1}e^{\zeta(\varphi_{\eps}-2h_{\eps})}C^{E\zeta}(e^{A\zeta}-e^{-\tilde{B}\zeta}) R_{M,r}^{-2}i\ddbar|f(z)|^2.
  \end{multline*}
  Since $e^{\zeta(\varphi_{\eps}-2h_{\eps})}>C^{-E\zeta}e^{-\tilde B\zeta}$ and $R_{M,r}<r$, this is bounded below by $\Theta(r)$, which is assumed to be positive definite on $U_\eps$.

  Fix $R>0$ satisfying $R_{M,r}<R<r$.  We may assume that $R>0$ has been chosen sufficiently small so that whenever $z\in K\cap U_\eps$ and $|f(z)|\leq R$, we have $f(z)\in U_{M,r}$ (by definition of $U_{M,r}$) and $\abs{\varphi_\eps(z)}\leq\frac{1}{4}(\tilde{A}+3\tilde{B})$ (since $|\varphi_\eps(z)|\leq\tilde{B}$ when $z\in\dom\phi$ and $|f(z)|\leq R_{M,r}$).  When $z\in\dom\phi$ and $|f(z)|>R_{M,r}$, we know $\varphi_\eps(z)<-A$.  By construction, if we apply the uniform bound $|f(z)|\geq R$ for $z\in\dom\phi$, we obtain a uniform bound $\varphi_\eps(z)\leq -A_R$ for some $A_R>A$. For each $j\in J$, we define
  \[
    \mathcal{O}_j=\set{z\in U_j:|f(z)|<R\text{ or }\varphi_j(z)>-\frac{1}{2}(A_R+A)}.
  \]
  For a constant $\sigma>0$ to be determined later, we define
  \[
    \mathcal{O}_\eps=\set{z\in\dom\phi:|f(z)|<R\text{ and }\dist(\partial U_\eps,z)>\sigma}.
  \]
  Recall that $r>0$ was chosen sufficiently small so that if $|f(z)|<r$ then either $z\notin K\cap\partial U_\eps$ or there exists $j\in J$ such that $z\in U_j$ and $\varphi_j(z)>\frac{1}{2}(\tilde A+\tilde B)$.  Consequently,
  \[
    \set{z\in K\cap\partial U_\eps:|f(z)|\leq R}\subset\bigcup_{j\in J}\set{z\in U_j:\varphi_j(z)>\frac{1}{2}(\tilde A+\tilde B)}.
  \]
  Hence, we may choose $\sigma>0$ sufficiently small so that
  \begin{equation}
  \label{eq:K_boundary_subset}
    \set{z\in K\cap\partial\mathcal{O}_\eps:|f(z)|<R}\subset\bigcup_{j\in J}\set{z\in \mathcal{O}_j:\varphi_j(z)>\frac{1}{2}(\tilde A+\tilde B)}.
  \end{equation}
  Now we may define
  \[
    \tilde U_V=\mathcal{O}_\eps\cup\bigcup_{j\in J}\mathcal{O}_j,
  \]
  and note that this is necessarily a neighborhood of $\bar V$.  From \eqref{eq:K_boundary_subset}, we have
  \[
    K\cap\partial\tilde{U}_V\subset\{z\in\overline{\mathcal{O}_\eps}:|f(z)|=R\}\cup\bigcup_{j\in J}\partial \mathcal{O}_j.
  \]
  If $z\in K\cap\partial\tilde{U}_V$ and $z\in\partial\mathcal{O}_\eps$, then \eqref{eq:K_boundary_subset} implies that $|f(z)|=R$, so $\varphi_\eps(z)<-A$.  On the other hand, if $z\in K\cap\partial\tilde{U}_V$ and $z\in\partial \mathcal{O}_j$ for some $j\in J$, then there are two possibilities. Either $z\in \partial U_j$ as well, in which case $\varphi_j(z)\leq-\tilde{A}<-A$, or $z\in U_j\backslash\mathcal{O}_j$, in which case $\varphi_j(z)\leq -\frac{1}{2}(A_R+A)<-A$ by construction.

  Our goal is to use Lemma \ref{lem:patching} to define $\varphi_V$ on the closure of
  \[
    U_V=\set{z\in\tilde{U}_V:\dist(K,z)<\xi\text{ and }\dist(\partial\tilde U_V,z)>\xi},
  \]
  where $\xi>0$ is a small constant to be determined later.  We assume that $\xi$ is at least small enough so that $\overline{V}\subset U_V$.  Furthermore, we may assume that $\xi$ is sufficiently small so that $\varphi_\eps<-A$ on $K\cap\partial U_V\cap\mathcal{O}_\eps$ and $\varphi_j<-A$ on $K\cap\partial U_V\cap\mathcal{O}_j$ for any $j\in J$.  Fix $j\in J$ and $z\in K\cap\partial \mathcal{O}_j\cap\mathcal{O}_\eps$.  Since this means $|f(z)|<R$, we must have $z\in\partial U_j$ as well, and hence $\varphi_j(z)\leq-\tilde{A}<-\frac{1}{4}(\tilde{A}+3\tilde{B})\leq\varphi_{\eps}(z)$.  Therefore, we may choose $\xi$ sufficiently small so that $\varphi_j<\varphi_\eps$ on $\overline{U_V}\cap\partial \mathcal{O}_j\cap\mathcal{O}_\eps$ for all $j\in J$.  Next, we fix $j\in J$ and $z\in K\cap\partial \mathcal{O}_\eps\cap \mathcal{O}_j$.  If $|f(z)|<R$, then \eqref{eq:K_boundary_subset} gives us $\varphi_{\eps}(z)\leq\frac{1}{4}(\tilde{A}+3\tilde{B})<\frac{1}{2}(\tilde{A}+\tilde{B})<\varphi_{j}(z)$. If $|f(z)|=R$, then $\varphi_\eps(z)\leq-A_R<-\frac{1}{2}(A_R+A)<\varphi_j(z)$.  Hence, we may choose $\xi$ sufficiently small so that $\varphi_\eps<\varphi_j$ on $\overline{U_V}\cap\mathcal{O}_j\cap\partial\mathcal{O}_\eps$ for all $j\in J$.

  For any $j\in J$, choose a neighborhood $\tilde U_j$ of $\overline{U_j}$ such that $\{\tilde U_j\}_{j\in J}$ is still disjoint, and let $\mathcal{O}_{j\eps}=U_V\cap\tilde U_j$.  We may now use Lemma \ref{lem:patching} to construct $\varphi_{j\eps}$ on $\mathcal{O}_{j\eps}$ by patching $\varphi_j$ on $\mathcal{O}_j$ and $\varphi_\eps$ on $\mathcal{O}_\eps$.  Since $\varphi_{j\eps}=\varphi_\eps$ on $\mathcal{O}_{j\eps}\backslash\overline{\mathcal{O}_j}$, we may define a smooth function on $\overline{U_V}$ by $\varphi_V=\varphi_{j\eps}$ on $\mathcal{O}_{j\eps}$ for every $j\in J$ and $\varphi_V=\varphi_\eps$ on $\overline{U_V}\backslash\bigcup_{j\in J}\mathcal{O}_{j\eps}$.
\end{proof}

\section{The Global Construction}
\label{sec:global_construction}

We first construct a weight function covering the set of infinite type points.

\begin{lem}
\label{lem:weight_near_K}
  Let $\Omega\subset\mathbb{C}^n$ be a smooth, bounded, pseudoconvex domain admitting a family of good vector fields.  If $K\subset\partial\Omega$ is the set of infinite type points, suppose that every $p\in K$ belongs to an admissible leaf of $K$.  If $0<\eta<1$ and $B>2(n-1)\log C$ where $C$ is the constant given in Lemma \ref{lem:h_V_construction}, then there exists a neighborhood $U_K$ of $K$ and a smooth, real-valued function $\varphi_K$ on $U_K$ such that $\abs{\varphi_K}\leq B$ on $U_K$ and
  \begin{equation}
  \label{eq:weight_near_K}
    i\ddbar\varphi_K+2\beta-i\frac{\eta}{1-\eta}(\partial\varphi_K-2\pi_{1,0}\alpha)\wedge(\dbar\varphi_K-2\pi_{0,1}\alpha)>0.
  \end{equation}
\end{lem}

\begin{rem}
  Note that Liu has recently shown that a condition closely related to \eqref{eq:weight_near_K} is both necessary and sufficient for the existence of a defining function $\rho$ such that $-(-\rho)^\eta$ is plurisubharmonic on $\Omega$ \cite{Liu17a}.  We have found it helpful to construct $\varphi_K$ so that \eqref{eq:weight_near_K} holds in all directions, but Liu has shown that it suffices to restrict this estimate to the weakly pseudoconvex directions.
\end{rem}

\begin{proof}
  Since $K$ is compact by a result of D'Angelo \cite{DAn82}, we may cover $K$ with a finite collection of neighborhoods $U_V$ given by Lemma \ref{lem:weight_near_leaf}, and apply Lemma \ref{lem:patching} to patch the functions $\varphi_V$ together to obtain a global function $\varphi_K$.
\end{proof}

With the points of infinite type under control, we can easily extend our construction to all of $\partial\Omega$.
\begin{lem}
\label{lem:weight_near_boundary}
  Let $\Omega\subset\mathbb{C}^n$ be a smooth, bounded, pseudoconvex domain and suppose that there exists a neighborhood $U_K$ of the set of infinite type points $K\subset\partial\Omega$ and a smooth, bounded, real-valued function $\varphi_K$ on $U_K$ satisfying \eqref{eq:weight_near_K} on $U_K$ for some $0<\eta<1$.  Then there exists a neighborhood $U$ of $\partial\Omega$ and a smooth, bounded, real-valued function $\varphi$ on $U$ such that
  \begin{equation}
  \label{eq:weight_near_boundary}
    i\ddbar\varphi+2\beta-i\frac{\eta}{1-\eta}(\partial\varphi-2\pi_{1,0}\alpha)\wedge(\dbar\varphi-2\pi_{0,1}\alpha)>0.
  \end{equation}
\end{lem}

\begin{proof}
  Fix $B>A>\sup_{U_K}\abs{\varphi_K}$.  Using Lemma \ref{lem:weight_near_point} with Remark \ref{rem:finite_type}, observe that every $p\in\partial\Omega\backslash K$ admits a neighborhood $U_p$ and a weight function $\varphi_p$ such that $\varphi_p(p)\geq A$, $\varphi_p\leq-A$ on $\partial U_p\cap\partial\Omega$, $\abs{\varphi_p}<B$ on $\overline{U_p}$, and \eqref{eq:weight_near_point} holds.  Since the neighborhoods on which $\varphi_p>\sup_{U_K}\abs{\varphi_K}$ cover the compact set $\partial\Omega\backslash U_K$, we may reduce to a finite subcover and use Lemma \ref{lem:patching} to construct a global weight $\varphi$.
\end{proof}

\begin{lem}
\label{lem:DF_index}
  Let $\Omega\subset\mathbb{C}^n$ be a smooth, bounded, pseudoconvex domain and suppose that there exists a neighborhood $U$ of $\partial\Omega$ and a smooth, bounded, real-valued function $\varphi$ on $U$ satisfying \eqref{eq:weight_near_boundary} on $U$ for some $0<\eta<1$.  Then there exists a neighborhood $\tilde U$ of $\partial\Omega$ such that $-e^{-\eta\varphi}(-\delta)^\eta$ is strictly plurisubharmonic on $\Omega\cap\tilde{U}$.
\end{lem}

\begin{proof}
  Let $\rho=e^{-\varphi}\delta$.  In order for \eqref{eq:weight_near_boundary} to hold, we must assume that $U$ is sufficiently small so that $\delta$ is smooth on $U$.  By shrinking $U$, we may further assume that the strict inequality in \eqref{eq:weight_near_boundary} holds on $\overline{U}$.  On $U$, we compute
  \[
    d\rho=e^{-\varphi}(d\delta-\delta d\varphi)
  \]
  and
  \[
    i\ddbar\rho=ie^{-\varphi}(\ddbar\delta+(-\delta)\ddbar\varphi-\partial\varphi\wedge\dbar\delta-\partial\delta\wedge\dbar\varphi-(-\delta)\partial\varphi\wedge\dbar\varphi).
  \]
  Hence on $U\cap\Omega$,
  \begin{multline*}
    \Theta:=ie^{\varphi}(\ddbar\rho+(1-\eta)(-\rho)^{-1}\partial\rho\wedge\dbar\rho)=\\
    i(\ddbar\delta+(-\delta)\ddbar\varphi+(1-\eta)(-\delta)^{-1}\partial\delta\wedge\dbar\delta-\eta\partial\varphi\wedge\dbar\delta-\eta\partial\delta\wedge\dbar\varphi-\eta(-\delta)\partial\varphi\wedge\dbar\varphi).
  \end{multline*}
  Since $i\ddbar(-(-\rho)^\eta)=\eta(-\rho)^{\eta-1}e^{-\varphi}\Theta$, it suffices to show that $\Theta$ is positive definite on some internal neighborhood of $\partial\Omega$.

  Before proceeding, we derive a variation of the Weinstock formula \cite{Wei75}.  For $z\in U$, there exists a unique element $\xi(z)\in\partial\Omega$ satisfying $\abs{\xi(z)-z}=\dist(z,\partial\Omega)$ and $\xi(z)=z-2\delta(z)\frac{\partial\delta}{\partial\bar z}(z)$ (Theorem 4.8 (3) in \cite{Fed59}).  Hence,
  \begin{equation}
  \label{eq:Weinstock_integral}
    \frac{\partial^2\delta}{\partial z_j\partial\bar z_k}(z)=\frac{\partial^2\delta}{\partial z_j\partial\bar z_k}(\xi(z))-\int_0^{-\delta(z)}\frac{d}{dt}\frac{\partial^2\delta}{\partial z_j\partial\bar z_k}\left(z+2t\frac{\partial\delta}{\partial\bar z}(z)\right)dt.
  \end{equation}
  We have
  \begin{multline*}
    \frac{d}{dt}\frac{\partial^2\delta}{\partial z_j\partial\bar z_k}\left(z+2t\frac{\partial\delta}{\partial\bar z}(z)\right)=\sum_{\ell=1}^n2\frac{\partial\delta}{\partial\bar z_\ell}(z)\frac{\partial^3\delta}{\partial z_\ell\partial z_j\partial\bar z_k}\left(z+2t\frac{\partial\delta}{\partial\bar z}(z)\right)\\
    +\sum_{\ell=1}^n2\frac{\partial\delta}{\partial z_\ell}(z)\frac{\partial^3\delta}{\partial\bar z_\ell\partial z_j\partial\bar z_k}\left(z+2t\frac{\partial\delta}{\partial\bar z}(z)\right).
  \end{multline*}
  Since $\frac{\partial\delta}{\partial z}(z)=\frac{\partial\delta}{\partial z}\left(z+2t\frac{\partial\delta}{\partial\bar z}(z)\right)$ and
  \begin{multline*}
    0=\frac{\partial^2}{\partial z_j\partial\bar z_k}\sum_{\ell=1}^n\abs{\frac{\partial\delta}{\partial z_\ell}}^2=\sum_{\ell=1}^n\frac{\partial\delta}{\partial\bar z_\ell}\frac{\partial^3\delta}{\partial z_\ell\partial z_j\partial\bar z_k}+\sum_{\ell=1}^n\frac{\partial\delta}{\partial z_\ell}\frac{\partial^3\delta}{\partial\bar z_\ell\partial z_j\partial\bar z_k}\\
    +\sum_{\ell=1}^n\frac{\partial^2\delta}{\partial z_\ell\partial z_j}\frac{\partial^2\delta}{\partial\bar z_\ell\partial\bar z_k}+\sum_{\ell=1}^n\frac{\partial^2\delta}{\partial\bar z_\ell\partial z_j}\frac{\partial^2\delta}{\partial z_\ell\partial\bar z_k},
  \end{multline*}
  we obtain
  \begin{multline*}
    \frac{d}{dt}\frac{\partial^2\delta}{\partial z_j\partial\bar z_k}\left(z+2t\frac{\partial\delta}{\partial\bar z}(z)\right)=-2\sum_{\ell=1}^n\left(\frac{\partial^2\delta}{\partial z_\ell\partial z_j}\frac{\partial^2\delta}{\partial\bar z_\ell\partial\bar z_k}\right)\left(z+2t\frac{\partial\delta}{\partial\bar z}(z)\right)\\
    -2\sum_{\ell=1}^n\left(\frac{\partial^2\delta}{\partial\bar z_\ell\partial z_j}\frac{\partial^2\delta}{\partial z_\ell\partial\bar z_k}\right)\left(z+2t\frac{\partial\delta}{\partial\bar z}(z)\right).
  \end{multline*}
  For the sake of our estimates, it will suffice to note
  \begin{multline*}
    \frac{d}{dt}\frac{\partial^2\delta}{\partial z_j\partial\bar z_k}\left(z+2t\frac{\partial\delta}{\partial\bar z}(z)\right)\leq-2\sum_{\ell=1}^n\left(\frac{\partial^2\delta}{\partial z_\ell\partial z_j}\frac{\partial^2\delta}{\partial\bar z_\ell\partial\bar z_k}\right)(\xi(z))\\
    -2\sum_{\ell=1}^n\left(\frac{\partial^2\delta}{\partial\bar z_\ell\partial z_j}\frac{\partial^2\delta}{\partial z_\ell\partial\bar z_k}\right)(\xi(z))+O\left((-\delta(z)-t)\norm{\delta}_{C^3(U)}\norm{\delta}_{C^2(U)}\right).
  \end{multline*}
  Substituting in \eqref{eq:Weinstock_integral}, we have
  \begin{multline}
  \label{eq:Weinstock_estimate}
    \frac{\partial^2\delta}{\partial z_j\partial\bar z_k}(z)\geq\frac{\partial^2\delta}{\partial z_j\partial\bar z_k}(\xi(z))+2(-\delta(z))\sum_{\ell=1}^n\left(\frac{\partial^2\delta}{\partial z_\ell\partial z_j}\frac{\partial^2\delta}{\partial\bar z_\ell\partial\bar z_k}\right)(\xi(z))\\
    +2(-\delta(z))\sum_{\ell=1}^n\left(\frac{\partial^2\delta}{\partial\bar z_\ell\partial z_j}\frac{\partial^2\delta}{\partial z_\ell\partial\bar z_k}\right)(\xi(z))-O\left((-\delta(z))^2\norm{\delta}_{C^3(U)}\norm{\delta}_{C^2(U)}\right).
  \end{multline}

  Now, suppose $\tau$ is a unit length eigenvector of the Levi-form at $\xi(z)$ with eigenvalue $\mu$.  Then
  \[
    \sum_{j=1}^n\tau^j\frac{\partial^2\delta}{\partial z_j\partial\bar z_k}(\xi(z))=\mu\tau^k+4\sum_{j,\ell=1}^n\tau^j\left(\frac{\partial^2\delta}{\partial z_j\partial\bar z_\ell}\frac{\partial\delta}{\partial z_\ell}\frac{\partial\delta}{\partial\bar z_k}\right)(\xi(z)),
  \]
  where the final term represents the component of the complex Hessian which is not contained in the Levi-form.  However, if we view $\tau$ as a $(1,0)$ vector field, then this is simply
  \[
    \sum_{j=1}^n\tau^j\frac{\partial^2\delta}{\partial z_j\partial\bar z_k}(\xi(z))=\mu\tau^k+2\alpha(\tau)\frac{\partial\delta}{\partial\bar z_k}(\xi(z)).
  \]
  Using \eqref{eq:alpha_projection},
  \[
    \sum_{j,k,\ell=1}^n\tau^j\left(\frac{\partial^2\delta}{\partial z_\ell\partial z_j}\frac{\partial^2\delta}{\partial\bar z_\ell\partial\bar z_k}\right)\bar\tau^k=\beta(-i\tau\wedge\bar\tau)+\abs{\alpha(\tau)}^2.
  \]
  Since $\abs{\beta(\xi(z))-\beta(z)}<O\left((-\delta(z))\norm{\delta}_{C^3(U)}\norm{\delta}_{C^2(U)}\right)$ with a similar estimate for $\abs{\alpha(\xi(z))-\alpha(z)}$, these are admissible error terms, so we can use \eqref{eq:Weinstock_estimate} to obtain
  \begin{multline*}
    \sum_{j,k=1}^n\tau^j\frac{\partial^2\delta}{\partial z_j\partial\bar z_k}\bar\tau^k\geq\mu \ddbar\abs{z}^2(\tau\wedge\bar\tau)+2(-\delta(z))(\beta(-i\tau\wedge\bar\tau)+\abs{\alpha(\tau)}^2)\\
    +2(-\delta(z))\left(\mu^2 \ddbar\abs{z}^2(\tau\wedge\bar\tau)+\abs{\alpha(\tau)}^2\right)\\
    -O\left((-\delta(z))^2\norm{\delta}_{C^3(U)}\norm{\delta}_{C^2(U)}\right).
  \end{multline*}
  Pseudoconvexity guarantees that $\mu\geq 0$, so we can discard these terms and obtain
  \begin{multline}
  \label{eq:levi_form_eigenvalue}
    \sum_{j,k=1}^n\tau^j\frac{\partial^2\delta}{\partial z_j\partial\bar z_k}\bar\tau^k\geq (-\delta(z))(2\beta(-i\tau\wedge\bar\tau)+4\abs{\alpha(\tau)}^2)\\
    -O\left((-\delta(z))^2\norm{\delta}_{C^3(U)}\norm{\delta}_{C^2(U)}\right).
  \end{multline}

  Turning our attention to $\varphi$, we choose $\eps>0$ so that
  \[
    i\ddbar\varphi+2\beta-i\frac{\eta}{1-\eta}(\partial\varphi-2\pi_{1,0}\alpha)\wedge(\dbar\varphi-2\pi_{0,1}\alpha)\geq i\eps\ddbar\abs{z}^2
  \]
  on $U$.  Combining this with \eqref{eq:levi_form_eigenvalue}, we obtain
  \begin{multline*}
    \Theta(-i\tau\wedge\bar\tau)\geq(-\delta(z))\eps\ddbar\abs{z}^2(\tau\wedge\bar\tau)+4(-\delta(z))\abs{\alpha(\tau)}^2\\
    +\frac{\eta}{1-\eta}(-\delta(z))\abs{(\partial\varphi-2\alpha)(\tau)}^2-\eta(-\delta(z))\abs{\partial\varphi(\tau)}^2\\
    -O\left((-\delta(z))^2\norm{\delta}_{C^3(U)}\norm{\delta}_{C^2(U)}\right).
  \end{multline*}
  If we expand and simplify, this is equivalent to
  \begin{multline*}
    \Theta(-i\tau\wedge\bar\tau)\geq(-\delta(z))\eps \ddbar\abs{z}^2(\tau\wedge\bar\tau)
    +\frac{1}{1-\eta}(-\delta(z))\abs{(\eta\partial\varphi-2\alpha)(\tau)}^2\\
    -O\left((-\delta(z))^2\norm{\delta}_{C^3(U)}\norm{\delta}_{C^2(U)}\right).
  \end{multline*}
  Hence,
  \begin{multline}
  \label{eq:Theta_tangential_estimate}
    \Theta(-i\tau\wedge\bar\tau)-(-\delta(z))\frac{\eps}{2}\ddbar\abs{z}^2(\tau\wedge\bar\tau)\geq
    (-\delta(z))\frac{\eps}{2}\ddbar\abs{z}^2(\tau\wedge\bar\tau)\\
    +\frac{1}{1-\eta}(-\delta(z))\abs{(\eta\partial\varphi-2\alpha)(\tau)}^2
    -O\left((-\delta(z))^2\norm{\delta}_{C^3(U)}\norm{\delta}_{C^2(U)}\right).
  \end{multline}
  Let $\nu=4\sum_{j=1}^n\frac{\partial\delta}{\partial\bar z_j}\frac{\partial}{\partial z_j}$, so that $\partial\delta(\nu)=1$.  Then
  \begin{equation}
  \label{eq:Theta_mixed_estimate}
    \abs{\Theta(-i\tau\wedge\bar\nu)}\leq\abs{(2\alpha-\eta\partial\varphi)(\tau)}
    +O\left((-\delta(z))\left(\norm{\varphi}_{C^2(U)}+\norm{\varphi}_{C^1(U)}^2\right)\right).
  \end{equation}
  Finally,
  \begin{multline*}
    \Theta(-i\nu\wedge\bar\nu)\geq (1-\eta)(-\delta(z))^{-1}\\
    -O\left(\norm{\delta}_{C^2(U)}+\norm{\varphi}_{C^1(U)}+(-\delta(z))\left(\norm{\varphi}_{C^2(U)}+\norm{\varphi}_{C^1(U)}^2\right)\right),
  \end{multline*}
  so
  \begin{multline}
  \label{eq:Theta_normal_estimate}
    \Theta(-i\nu\wedge\bar\nu)-(-\delta(z))\frac{\eps}{2}\ddbar\abs{z}^2(\nu\wedge\bar\nu)\geq (1-\eta)(-\delta(z))^{-1}\\
    -O\left(\norm{\delta}_{C^2(U)}+\norm{\varphi}_{C^1(U)}+(-\delta(z))\left(\norm{\varphi}_{C^2(U)}+\norm{\varphi}_{C^1(U)}^2+\eps\right)\right).
  \end{multline}
  For $z\in\Omega$ and $|\delta(z)|$ sufficiently small, the right-hand sides of \eqref{eq:Theta_tangential_estimate} and \eqref{eq:Theta_normal_estimate} are strictly positive, so we may multiply these estimates and subtract the square of \eqref{eq:Theta_mixed_estimate} to obtain
  \begin{multline*}
    \left(\Theta(-i\tau\wedge\bar\tau)-(-\delta(z))\frac{\eps}{2}\ddbar\abs{z}^2(\tau\wedge\bar\tau)\right)\left(\Theta(-i\nu\wedge\bar\nu)
    -(-\delta(z))\frac{\eps}{2}\ddbar\abs{z}^2(\nu\wedge\bar\nu)\right)\\
    -\abs{\Theta(-i\tau\wedge\bar\nu)}^2\geq -O\left(-\delta(z)\right),
  \end{multline*}
  where the constant in the error term depends only on $\norm{\varphi}_{C^2(U)}$, $\norm{\delta}_{C^3(U)}$, and $\eta$.  We conclude that $\Theta\geq i\frac{\eps}{4}(-\delta(z))\ddbar\abs{z}^2$ for $z\in\Omega$ when $|\delta(z)|$ is sufficiently small.

\end{proof}

\begin{proof}[Proof of Theorem \ref{thm:Main} (1)]
  Let $\tilde U$ and $\varphi$ be given by Lemma \ref{lem:DF_index}.  Let
  \[
    A=-\sup_{\partial\tilde U\cap\Omega}e^{-\varphi}\delta>0
  \]
  and choose $B>0$ sufficiently small so that $B|z|^2-A<0$ on $\partial\Omega$.  Then we may define
  \[
    \rho(z)=\begin{cases}\max\set{e^{-\varphi(z)}\delta(z),B|z|^2-A}&z\in\tilde{U}\\
    B|z|^2-A&z\in\Omega\backslash\tilde{U}
    \end{cases}
  \]
  to obtain a Lipschitz defining function on $\tilde{U}\cup\Omega$ such that $-(-\rho)^\eta$ is strictly plurisubharmonic on $\Omega$.  A simple modification of Lemma \ref{lem:patching} can be used to obtain a smooth defining function.
\end{proof}

\section{Stein Neighborhood Bases}
\label{sec:Stein}

The proof of Theorem \ref{thm:Main} (2) is nearly identical to Theorem \ref{thm:Main} (1), so we merely outline the argument and highlight the differences.
\begin{lem}
\label{lem:Stein_weight_near_K}
  Let $\Omega\subset\mathbb{C}^n$ be a smooth, bounded, pseudoconvex domain admitting a family of good vector fields.  If $K\subset\partial\Omega$ is the set of infinite type points, suppose that every $p\in K$ belongs to an admissible leaf of $K$.  If $0<\eta<1$ and $B>2(n-1)\log C$ where $C$ is the constant given in Lemma \ref{lem:h_V_construction}, then there exists a neighborhood $U_K$ of $K$ and a smooth, real-valued function $\varphi_K$ on $U_K$ such that $\abs{\varphi_K}\leq B$ on $U_K$ and
  \begin{equation}
  \label{eq:Stein_weight_near_K}
    i\ddbar\varphi_K-2\beta-i\frac{1}{1-\eta}(\partial\varphi_K+2\pi_{1,0}\alpha)\wedge(\dbar\varphi_K+2\pi_{0,1}\alpha)>0.
  \end{equation}
\end{lem}

\begin{proof}
  The proof is nearly identical to the proof of Lemma \ref{lem:weight_near_K}, except we replace $h_{\eps}$ with $-h_{\eps}$ in the proof of Lemma \ref{lem:weight_near_leaf}.
\end{proof}

\begin{lem}
\label{lem:Stein_weight_near_boundary}
  Let $\Omega\subset\mathbb{C}^n$ be a smooth, bounded, pseudoconvex domain and suppose that there exists a neighborhood $U_K$ of the set of infinite type points $K\subset\partial\Omega$ and a smooth, bounded, real-valued function $\varphi_K$ on $U_K$ satisfying \eqref{eq:weight_near_K} on $U_K$ for some $0<\eta<1$.  Then there exists a neighborhood $U$ of $\partial\Omega$ and a smooth, bounded, real-valued function $\varphi$ on $U$ such that
  \begin{equation}
  \label{eq:Stein_weight_near_boundary}
    i\ddbar\varphi-2\beta-i\frac{1}{1-\eta}(\partial\varphi+2\pi_{1,0}\alpha)\wedge(\dbar\varphi+2\pi_{0,1}\alpha)>0.
  \end{equation}
\end{lem}

\begin{proof}
  Once again, the proof is nearly identical to the proof of Lemma \ref{lem:weight_near_boundary}.
\end{proof}

\begin{lem}
\label{lem:Stein_index}
  Let $\Omega\subset\mathbb{C}^n$ be a smooth, bounded, pseudoconvex domain and suppose that there exists a neighborhood $U$ of $\partial\Omega$ and a smooth, bounded, real-valued function $\varphi$ on $U$ satisfying \eqref{eq:weight_near_boundary} on $U$ for some $0<\eta<1$.  Then there exists a neighborhood $\tilde U$ of $\partial\Omega$ such that $e^{\varphi/\eta}\delta^{1/\eta}$ is strictly plurisubharmonic on $\Omega^c\cap\tilde{U}$.
\end{lem}

\begin{proof}
  If we set $\rho(z)=e^{\varphi(z)}\delta(z)$, then we have
  \[
    i\ddbar\rho^{1/\eta}=\eta^{-1}\rho^{1/\eta-1}(i\ddbar\rho+i(\eta^{-1}-1)\rho^{-1}\partial\rho\wedge\dbar\rho),
  \]
  so since
  \begin{multline*}
    i\ddbar\rho+i(\eta^{-1}-1)\rho^{-1}\partial\rho\wedge\dbar\rho=\\
    ie^{\varphi}(\ddbar\delta+\delta\ddbar\varphi+(\eta^{-1}-1)\delta^{-1}\partial\delta\wedge\dbar\delta+\eta^{-1}\partial\delta\wedge\dbar\varphi+\eta^{-1}\partial\varphi\wedge\dbar\delta+\eta^{-1}\delta\partial\varphi\wedge\dbar\varphi),
  \end{multline*}
  it suffices to show
  \[
    \Theta:=ie^{\varphi}(\ddbar\delta+\delta\ddbar\varphi+(\eta^{-1}-1)\delta^{-1}\partial\delta\wedge\dbar\delta+\eta^{-1}\partial\delta\wedge\dbar\varphi+\eta^{-1}\partial\varphi\wedge\dbar\delta+\eta^{-1}\delta\partial\varphi\wedge\dbar\varphi)
  \]
  is positive definite.  This will follow from \eqref{eq:Stein_weight_near_boundary} in exactly the same manner as in the proof of Lemma \ref{lem:DF_index}.
\end{proof}

\section{Examples}
\label{sec:examples}

Our first family of examples will consist of the model domains given in Corollary \ref{cor:cross_product}.
\begin{proof}[Proof of Corollary \ref{cor:cross_product}]
  Suppose $K$ satisfies the hypotheses of Corollary \ref{cor:cross_product}.  For every $p\in K$, we set $f(z)=g(z)-g(p)$, and $V=f^{-1}[\{0\}]\cap K$ will be an admissible leaf of dimension $m$.  The only property that requires some justification is Definition \ref{defn:leaf_m} \eqref{item:leaf_m_property_p}.  As shown by Boas in \cite{Boa88}, for every $1\leq j\leq n-m$ and $\tilde{M}>0$, there exists a smooth function $\lambda_j$ in a neighborhood of $g_j[K]$ such that $0\leq\lambda_j\leq 1$ and $i\ddbar\lambda_j\geq iM\ddbar\abs{z}^2$.  Given any $M>0$ and $r>0$, we choose any $0<R_{M,r}<r$, and set $\tilde{M}=\frac{(n-m)M}{R_{M,r}^2}$.  Then we may set $\lambda_{M,r}(z)=\frac{1}{n-m}\sum_{j=1}^{n-m}\lambda_j(z_j)$.  Having shown that $V$ is an admissible leaf, Corollary \ref{cor:cross_product} will follow from our main theorem.
\end{proof}

In contrast to Corollary \ref{cor:cross_product}, suppose $C\subset\mathbb{R}$ is a Smith-Volterra-Cantor set (i.e., topologically equivalent to the Cantor set, but with positive measure). Suppose that the set of points of infinite type in the boundary of some bounded, smooth, pseudoconvex domain is biholomorphic to $C\times C\subset\mathbb{C}$.  Then the methods of the present paper will fail, because the set contains no analytic discs but is too large to admit a good family of weight functions.  It would be of great interest to know whether such a domain could provide a counterexample to our main result, or if it could be approached with a more refined technique.  See Example 1 in \cite{Boa88} for a related example that could possibly be modified to provide a concrete construction of a smooth, bounded, pseudoconvex domain for which the methods of the present paper fail.

We also obtain a large family of examples from Catlin's Property $(P)$:
\begin{proof}[Proof of Corollary \ref{cor:Property_P}]
  Recall that a domain satisfies Property $(P)$ according to Catlin \cite{Cat84} if for every $M>0$ there exists a smooth, plurisubharmonic function $\lambda$ on $\overline\Omega$ such that $0\leq\lambda\leq 1$ and $i\ddbar\lambda\geq iM\ddbar\abs{z}^2$ on $\partial\Omega$.  In this case, every point of infinite type is an admissible leaf of dimension $0$, so the Corollary follows from our main theorem.
\end{proof}

For the remainder of this section, we will focus on examples with sets of infinite type points like those pictured in Figure \ref{fig:examples}, Example 1.  We will show that such domains admit a family of good vector fields, and hence the Diederich-Fornaess Index is equal to $1$.  Finally, we will construct a family of concrete examples adapted from the worm domain, in which case these results are nontrivial.

\begin{proof}[Proof of Proposition \ref{prop:annulus}]
  Set $p=(1,0,\ldots,0)$, so that after our biholomorphic change of coordinates $K\backslash\{p\}$ can be identified with $S\backslash\{1\}$.  We write $z_1=x+iy$.  Choose $r>0$ sufficiently small so that $z_1\in S\cap\overline{B(1,r)}$ only if $\abs{y}\geq\abs{x-1}^\gamma$.  We assume, without loss of generality, that $\gamma>\frac{1}{3}$, so that $2\gamma/(1-\gamma)>1$.

  Define $\chi\in C^\infty(\mathbb{R})$ to be a nondecreasing function satisfying $\chi(t)=0$ when $t\leq 0$ and $\chi(t)=1$ when $t\geq 1$.  For $\zeta>0$, we define
  \[
    \psi_{\zeta,m}(z_1)=\sum_{j=0}^m (x-1)^j\frac{(-i)^j}{j!}\frac{d^j}{dy^j}\chi(2-\zeta^{-2}y^2)
  \]
  on the set
  \[
    \mathcal{O}_{\zeta}=\set{z_1\in B(1,r):|y|<2\zeta,|x-1|<2(2\zeta)^{1/\gamma}}.
  \]
  When $\zeta$ is small enough so that $\overline{\mathcal{O}_\zeta}\subset B(1,r)$, we have $z_1\in\partial\mathcal{O}_\zeta\cap S$ only if $|y|=2\zeta$ and $|x-1|\leq(2\zeta)^{1/\gamma}$.  Note that we may choose $\chi$ to be real analytic on the interval $(0,1)$, in which case $\psi_{\zeta,m}$ will converge to a holomorphic extension of $\chi(2+\zeta^{-2}z_1^2)|_{z_1=iy}$, but the domain of convergence will be too small for our purposes when $z_1$ is close to $i\zeta$ or $i\sqrt{2}\zeta$.  Hence, we instead compute
  \[
    \frac{\partial}{\partial\bar z_1}\psi_{\zeta,m}(z_1)=-\frac{1}{2}(x-1)^m\frac{(-i)^{m+1}}{m!}\frac{d^{m+1}}{dy^{m+1}}\chi(2-\zeta^{-2}y^2).
  \]
  Note that derivatives of $\chi$ are only non-vanishing when $\zeta<\abs{y}<\sqrt{2}\zeta$.  Hence, we may inductively check that $\abs{\frac{d^j}{dy^j}\chi(2-\zeta^{-2}y^2)}\leq O(\zeta^{-j})$ when $j\geq 1$.  Throughout this paragraph, our error terms will be computed as $\zeta\rightarrow 0$.  On $\mathcal{O}_\zeta$, we have
  \[
    \abs{\frac{\partial}{\partial\bar z_1}\psi_{\zeta,m}(z_1)}\leq O\left(\zeta^{m/\gamma-m-1}\right)
  \]
  and
  \[
    \abs{\psi_{\zeta,m}(z_1)-\chi(2-\zeta^{-2}y^2)}\leq O(\zeta^{1/\gamma-1})
  \]
  Given $\eps>0$, we set $\zeta=\eps^{2\gamma/(1-\gamma)}$ and choose $m\geq \frac{1}{1-\gamma}$.  For such values, we will write $\psi_\eps=\psi_{\zeta,m}$ and $\mathcal{O}_\eps=\mathcal{O}_\zeta$.  We now have a function that vanishes on $\partial\mathcal{O}_\eps\cap S$ when $\eps$ is sufficiently small and satisfies $\abs{\frac{\partial}{\partial\bar z_1}\psi_{\eps}(z_1)}\leq O(\eps^2)$ and $\abs{\psi_{\eps}(z_1)-\chi(2-\eps^{-4\gamma/(1-\gamma)}y^2)}\leq O(\eps^2)$ on $\mathcal{O}_\eps$.

  By adapting the argument of \cite{BoSt93}, there exists an open set $\tilde U_\eps$ such that $K\backslash\{p\}\subset\tilde U_\eps$, a $(1,0)$ vector field $\tilde X_\eps$ with smooth coefficients on $\tilde U_\eps$, and a constant $C>0$ such that $2C^{-1}<\abs{\tilde X_\eps\delta}<\frac{1}{2}C$, $\abs{\arg\tilde X_\eps\delta}<\frac{1}{2}\eps$, and $\abs{\partial\delta([\tilde X_\eps,\partial/\partial\bar z_j])}<\frac{1}{2}\eps$ for all $1\leq j\leq n$, all on $\tilde U_\eps$.  As in Section \ref{sec:de_rham_cohomology}, the construction of $\tilde{X}_\eps$ relies on integrating the closed one-form $\alpha$ on $S$, so the simple-connectedness of $S\backslash\{1\}$ is sufficient for our purposes.  Assume that $\eps>0$ is sufficiently small so that $\overline{\mathcal{O}_\eps}\subset B(1,r)\cap B(1,3\eps^{2\gamma/(1-\gamma)})$ and $\delta$ is smooth on $\overline{B(p,3\eps^{2\gamma/(1-\gamma)})}$, and define
  \begin{multline*}
    U_\eps=\set{z\in\tilde U_\eps:|z_1-1|>r\text{ or }|\im z_1|>\sqrt{2}\eps^{2\gamma/(1-\gamma)}\text{ or }z_1\in\mathcal{O}_\eps}\\
    \cup\set{z\in B(p,3\eps^{2\gamma/(1-\gamma)}):|\im z_1|<\eps^{2\gamma/(1-\gamma)}\text{ and }z_1\in\mathcal{O}_\eps}.
  \end{multline*}
  On $U_\eps$, we define $X_\eps=\tilde X_\eps$ when $z\in U_\eps\cap\tilde U_\eps$ but $z_1\notin\mathcal{O}_\eps$, we define
  \[
    X_\eps=(1-\psi_{\eps}(z_1))\tilde X_\eps+4\psi_{\eps}(z_1)\sum_{j=1}^n\frac{\partial\delta}{\partial\bar z_j}(p)\frac{\partial}{\partial z_j}
  \]
  when $z\in U_\eps\cap\tilde U_\eps$ and $z_1\in\mathcal{O}_\eps$, and we define $X_\eps=4\sum_{j=1}^n\frac{\partial\delta}{\partial\bar z_j}(p)\frac{\partial}{\partial z_j}$ when $z\in U_\eps\backslash\tilde U_\eps$ and $z_1\in\mathcal{O}_\eps$.  For $z\in U_\eps\cap\tilde U_\eps$ satisfying $z_1\in\mathcal{O}_\eps$, we have
  \[
    X_\eps\delta(z)=(1-\psi_{\eps}(z_1))\tilde X_\eps\delta(z)+4\psi_{\eps}(z_1)\sum_{j=1}^n\frac{\partial\delta}{\partial\bar z_j}(p)\frac{\partial\delta}{\partial z_j}(z)
  \]
  so
  \begin{multline*}
    X_\eps\delta(z)=(1-\chi(2-\eps^{-4\gamma/(1-\gamma)}y^2))\tilde X_\eps\delta(z)+\chi(2-\eps^{-4\gamma/(1-\gamma)}y^2)\\
    +O\left(\eps^2+\norm{\delta}_{C^2(U_\eps)}\eps^{2\gamma/(1-\gamma)}\right).
  \end{multline*}
  When $\eps>0$ is sufficiently small, we may guarantee that on $U_\eps$ we have $C^{-1}<\abs{X_\eps\delta}<C$ and $\abs{\arg X_\eps\delta}<\eps$.  Furthermore,
  \[
    \partial\delta([X_\eps,\partial/\partial\bar z_1])=(1-\psi_{\eps}(z_1))\partial\delta([\tilde X_\eps,\partial/\partial\bar z_1])+\frac{\partial\psi_{\eps}}{\partial\bar z_1}(\tilde X_\eps\delta-1),
  \]
  so
  \[
    \abs{\partial\delta([X_\eps,\partial/\partial\bar z_1])}\leq\frac{1}{2}\eps+O(\eps^2).
  \]
  Once again, when $\eps>0$ is sufficiently small we have $\abs{\partial\delta([X_\eps,\partial/\partial\bar z_1])}<\eps$.  The proof when either $z_1\notin\mathcal{O}_\eps$ or $z\notin\tilde U_\eps$ is similar.
\end{proof}

With Proposition \ref{prop:annulus}, Corollary \ref{cor:annulus} will now follow from the main theorem.
\begin{proof}[Proof of Corollary \ref{cor:annulus}]
  As before, we let $p=(1,0,\ldots,0)$.  Given $M>0$ and $r>0$, choose
  \[
    0<R_{M,r}<\min\set{r,(2M)^{-\frac{\gamma}{2(1-\gamma)}}}
  \]
  sufficiently small so that $\abs{\re z_1-1}\leq\abs{\im z_1}^{1/\gamma}$ when $z\in K\cap\overline{B(p,R_{M,r})}$.  Set
  \[
    U_{M,r}=\set{z\in B(p,r):\sqrt{2(\re z_1-1)^2+\sum_{j=2}^n|z_j|^2}<\frac{R_{M,r}}{\sqrt{M}}}
  \]
  and
  \[
    \lambda_{M,r}(z)=\frac{2M}{R_{M,r}^2}(\re z_1-1)^2+\frac{M}{R_{M,r}^2}\sum_{j=2}^n|z_j|^2.
  \]
  It is easy to check that $0\leq\lambda_{M,r}\leq 1$ and $i\ddbar\lambda_{M,r}\geq i\frac{M}{R_{M,r}^2}\ddbar\abs{z}^2$ on $U_{M,r}$.  When $z\in K\cap\overline{B(p,R_{M,r})}$, we have $z_j=0$ for $2\leq j\leq n$ and
  \[
    \abs{\re z_1-1}\leq \abs{\im z_1}^{1/\gamma}\leq R_{M,r}^{1/\gamma}<\frac{R_{M,r}}{\sqrt{2M}}.
  \]
  Hence, $K\cap\overline{B(p,R_{M,r})}\subset U_{M,r}$, so we have shown that $p$ is an admissible leaf of dimension $0$ and $K\backslash\set{p}$ is an admissible leaf of dimension $1$.  The main theorem will imply that the Diederich-Fornaess Index of $\Omega$ is equal to $1$.
\end{proof}

We now construct an explicit example of a domain requiring Corollary \ref{cor:annulus}, modeled on the Diederich-Fornaess Worm Domain \cite{DiFo77a}.  We first prove that a general family of such domains exists:
\begin{lem}
\label{lem:worm_like_domain}
  Let $\mu:\mathbb{C}\rightarrow\mathbb{R}$ be a smooth function such that
  \begin{enumerate}
    \item $\mu(0)>0$

    \item $\liminf_{|z|\rightarrow\infty}\mu(z)>0$

    \item $\mu(e^{i\theta})\leq 0$ for all $\theta\in\mathbb{R}$

    \item There exist $\eps>0$ and $m>0$ such that $(\mu(z))^m\frac{\partial^2\mu}{\partial z\partial\bar z}(z)+\abs{\frac{\partial\mu}{\partial z}(z)}^2>0$ and $\abs{\frac{\partial\mu}{\partial z}(z)}>\frac{(\mu(z))^m}{|z|}$ whenever $0<\mu(z)<\eps$.
  \end{enumerate}
  Then there exists a domain $\Omega\subset\mathbb{C}^2$ satisfying
  \begin{enumerate}
    \item $\Omega$ is a smooth, bounded, pseudoconvex domain with the set of weakly pseudoconvex points given by
    \[
      K=\{(z_1,z_2)\in\mathbb{C}^2:z_1=0, \mu(z_2)\leq 0\}.
    \]

    \item If $X$ is a $(1,0)$ vector field with smooth coefficients in a neighborhood of $K$ such that $\abs{X\delta}>C^{-1}$ and $\arg X\delta=0$ on $K$, then there exists $z\in K$ at which $\abs{\partial\delta([X,\partial/\partial\bar z_2])}>\frac{1}{\sqrt{2}}C^{-1}$.

    \item $\Omega$ does not admit a $C^2$ defining function that is plurisubharmonic on $\partial\Omega$.

  \end{enumerate}
\end{lem}

\begin{rem}
\label{rem:strong_good_vector_field}
  Conclusion $(2)$ is significant because it shows that $\abs{\arg X_\eps\delta}<\eps$ is the best that we can hope for in Definition \ref{defn:good_vector_field}.  In particular, the methods of \cite{StSu02} do not apply, since these assume $\arg X_\eps\delta=0$.  See also Theorem 5.22 in \cite{Str10}, which works around this issue by introducing a complex-valued function with the same properties as a defining function except that the imaginary part is possibly nontrivial but smaller than $\eps>0$.
\end{rem}

\begin{proof}
  We may assume $m>1$.  Let $\chi(t)=\begin{cases}0&t\leq 0\\\exp(-1/t^{m-1})&t>0\end{cases}$.  Fix $A>0$ sufficiently small so that
  \begin{enumerate}
    \item $\mu(0)>A$,
    \item $(\mu(z))^m\frac{\partial^2\mu}{\partial z\partial\bar z}(z)+\abs{\frac{\partial\mu}{\partial z}(z)}^2>0$ and $\abs{\frac{\partial\mu}{\partial z}(z)}>\frac{(\mu(z))^m}{|z|}$ whenever $0<\mu(z)<A$, and
    \item $\{z\in\mathbb{C}:\mu(z)<A\}$ is bounded.
  \end{enumerate}
  For any $B>\left(\frac{mA^{m-1}+3}{m-1}\right)^{1/(m-1)}$, define
  \begin{equation}
  \label{eq:worm_like_definition}
    \rho(z_1,z_2)=\abs{z_1+\exp(i\log|z_2|^2)}^2-1+e^{A^{1-m}B^{m-1}}\chi(B^{-1} \mu(z_2)),
  \end{equation}
  and let $\Omega$ be the domain defined by $\rho$.  Note that $\rho>0$ whenever $\mu(z_2)>A$.

  On $K$, $z_1=0$ and $\chi(B^{-1} \mu(z_2))$ vanishes together with all of its derivatives, so we have
  \[
    \dbar\rho|_K=e^{i(\log|z_2|^2)}d\bar z_1\Rightarrow i\ddbar\rho|_K\equiv 0\pmod{\partial\rho\wedge\bar\theta,\theta\wedge\dbar\rho},
  \]
  where $\theta$ represents a $(1,0)$-form.  At the other extreme, when $\mu(z_2)=A$, we have $z_1=-\exp(i\log|z_2|^2)$, so
  \begin{multline*}
    \dbar\rho=e^{A^{1-m}B^{m-1}}B^{-1}\chi'(B^{-1}A)\frac{\partial \mu}{\partial\bar z_2}(z_2)d\bar z_2\\
    \Rightarrow i\ddbar\rho\equiv idz_1\wedge d\bar z_1\pmod{\partial\rho\wedge\bar\theta,\theta\wedge\dbar\rho},
  \end{multline*}
  where $\theta$ represents a $(1,0)$-form.  Hence, $\Omega$ is smooth and strictly pseudoconvex in a neighborhood of these points.

  For the remaining points, observe that the boundary of $\Omega$ is parameterized by the following: for any $z_2\in S$, where $S=\{z\in\mathbb{C}:0<\mu(z)<A\}$, and $\theta\in\mathbb{R}$, we have $z_1=e^{i\log|z_2|^2}(r(z_2)e^{i\theta}-1)$, where $r(z_2)=\sqrt{1-e^{A^{1-m}B^{m-1}}\chi(B^{-1} \mu(z_2))}$.  At such points, we may compute
  \[
    \dbar\rho=re^{i(\theta+\log|z_2|^2)}d\bar z_1
    -\left(\frac{ie^{i\theta}-ie^{-i\theta}}{\bar z_2}+2\frac{\partial r}{\partial\bar z_2}\right)r d\bar z_2.
  \]
  and
  \begin{multline*}
    i\ddbar\rho=i\left(dz_1+\frac{i}{z_2}\exp(i\log|z_2|^2)dz_2\right)\wedge\left(d\bar z_1-\frac{i}{\bar z_2}\exp(-i\log|z_2|^2)d\bar z_2\right)+\\
    +i\left(\frac{1}{\abs{z_2}^2}-\frac{(e^{i\theta}+e^{-i\theta})r}{\abs{z_2}^2}-2r\frac{\partial^2 r}{\partial z_2\partial\bar z_2}-2\abs{\frac{\partial r}{\partial\bar z_2}}^2\right)dz_2\wedge d\bar z_2.
  \end{multline*}
  Hence, $L=e^{i(\theta+\log|z_2|^2)}\left(\frac{ie^{i\theta}-ie^{-i\theta}}{z_2}+2\frac{\partial r}{\partial z_2}\right)\frac{\partial}{\partial z_1}+\frac{\partial}{\partial z_2}$ spans the tangential $(1,0)$ vector fields, and since
  \[
    \left(dz_1+\frac{i}{z_2}\exp(i\log|z_2|^2)dz_2\right)(L)=\left(\frac{ie^{i\theta}}{z_2}+2\frac{\partial r}{\partial z_2}\right)e^{i(\theta+\log|z_2|^2)},
  \]
  the Levi-form is equal to
  \[
    \mathcal{L}(L,\bar L)=\frac{2ie^{i\theta}}{z_2}\frac{\partial r}{\partial\bar z_2}-\frac{2ie^{-i\theta}}{\bar z_2}\frac{\partial r}{\partial z_2}
    +\frac{2}{\abs{z_2}^2}-\frac{(e^{i\theta}+e^{-i\theta})r}{\abs{z_2}^2}-2r\frac{\partial^2 r}{\partial z_2\partial\bar z_2}+2\abs{\frac{\partial r}{\partial z_2}}^2.
  \]
  If we minimize this with respect to $\theta$, we have
  \begin{multline*}
    \mathcal{L}(L,\bar L)\geq-2\abs{\frac{2i}{\bar z_2}\frac{\partial r}{\partial z_2}+\frac{r}{\abs{z_2}^2}}
    +\frac{2}{\abs{z_2}^2}-2r\frac{\partial^2 r}{\partial z_2\partial\bar z_2}+2\abs{\frac{\partial r}{\partial z_2}}^2\\
    \geq-2r\frac{\partial^2 r}{\partial z_2\partial\bar z_2}+2\abs{\frac{\partial r}{\partial z_2}}^2-\frac{4}{\abs{z_2}}\abs{\frac{\partial r}{\partial z_2}}+\frac{2(1-r)}{\abs{z_2}^2}.
  \end{multline*}

  When $t\neq 0$, we have $\chi'(t)=\frac{(m-1)\chi(t)}{t^m}$, so we may compute
  \[
    \frac{\partial r}{\partial z_2}=\frac{-(m-1)B^{m-1}(r^{-1}-r)}{2\mu^m}\frac{\partial\mu}{\partial z_2}
  \]
  and hence
  \begin{multline*}
    \frac{\partial^2 r}{\partial z_2\partial\bar z_2}=\\
    \frac{-(m-1)B^{m-1}(r^{-1}-r)}{2\mu^m}\left(\left(\frac{(m-1)B^{m-1}(r^{-2}+1)}{2\mu^m}-\frac{m}{\mu}\right)\abs{\frac{\partial\mu}{\partial z_2}}^2+\frac{\partial^2\mu}{\partial z_2\partial\bar z_2}\right).
  \end{multline*}
  This gives us
  \begin{multline*}
    \frac{\mathcal{L}(L,\bar L)}{1-r}\geq\frac{2}{\abs{z_2}^2}+\frac{(m-1)B^{m-1}(1+r)}{\mu^m}\times\\
    \left(\left(\frac{(m-1)B^{m-1}}{r^2\mu^m}-\frac{m}{\mu}\right)\abs{\frac{\partial\mu}{\partial z_2}}^2-\frac{2}{r\abs{z_2}}\abs{\frac{\partial\mu}{\partial z_2}}+\frac{\partial^2\mu}{\partial z_2\partial\bar z_2}\right).
  \end{multline*}
  By our hypotheses on $\mu$, we have
  \begin{multline*}
    \frac{\mathcal{L}(L,\bar L)}{1-r}>\frac{2}{\abs{z_2}^2}+ \\ \frac{(m-1)B^{m-1}(1+r)}{\mu^m}\left(\frac{(m-1)B^{m-1}}{r^2\mu^m}-\frac{m}{\mu}-\frac{2}{r\mu^m}-\frac{1}{\mu^m}\right)\abs{\frac{\partial\mu}{\partial z_2}}^2,
  \end{multline*}
  and our lower bound on $B$ is sufficient to guarantee that this is positive definite whenever $0<\mu<A$ and $0<r<1$.

  On $K$, we have $\partial\rho=\exp(-i\log|z_2|^2)dz_1$, so $\partial\delta=\frac{1}{\sqrt{2}}\exp(-i\log|z_2|^2)dz_1$.  Let $X=X^1\frac{\partial}{\partial z_1}+X^2\frac{\partial}{\partial z_2}$.  If $\arg X\delta=0$ on $K$, then $X^1\exp(-i\log|z_2|^2)$ is real on $K$.  We write $X^1(0,z_2)=\exp(i\log|z_2|^2)\exp(Y(z_2))$ for some real, smooth, function $Y(z_2)$ bounded below by $-\log C$ defined whenever $\mu(z_2)\leq 0$.  On $K$,
  \[
    \partial\delta([X,\partial/\partial\bar z_2])=-\frac{1}{\sqrt{2}}\exp(-i\log|z_2|^2)\frac{\partial X^1}{\partial\bar z_2}=-\frac{1}{\sqrt{2}}e^Y\left(\frac{\partial Y}{\partial\bar z_2}+\frac{i}{\bar z_2}\right).
  \]
  Note that $Y$ is defined when $z_2=e^{i\theta}$ for any $\theta\in\mathbb{R}$, so we can compute $\frac{\partial}{\partial\theta}Y(e^{i\theta})=2\re\left(ie^{i\theta}\frac{\partial Y}{\partial z_2}(e^{i\theta})\right)$.  Since $Y(e^0)=Y(e^{2\pi i})$, there must exist $0<\theta_0<2\pi$ at which $\frac{\partial}{\partial\theta}Y(e^{i\theta})|_{\theta=\theta_0}=0$.  At such a point, there must exist $t\in\mathbb{R}$ such that $\frac{\partial Y}{\partial z_2}(e^{i\theta_0})=e^{-i\theta_0}t$.  Hence,
  \[
    \partial\delta([X,\partial/\partial\bar z_2])|_{(0,e^{i\theta_0})}=-\frac{1}{\sqrt{2}}e^Ye^{i\theta_0}(t+i),
  \]
  and we have $\abs{\partial\delta([X,\partial/\partial\bar z_2])|_{(0,e^{i\theta_0})}}>\frac{1}{\sqrt{2}}C^{-1}$

  The proof that $\Omega$ fails to admit a plurisubharmonic defining function is identical to the proof for the worm domain given in \cite{DiFo77a}.
\end{proof}

To construct a family of explicit examples, we first consider the set illustrated by Example 1 in Figure \ref{fig:examples}.  For any $0<s<1$, we define
\begin{multline*}
  \mu(z)=(|z|^2-1)^2-s^2|z-1|^4\\
  =4(\re z-1)^2+4(\re z-1)|z-1|^2+(1-s^2)|z-1|^4.
\end{multline*}
Since $|z|^2-1\pm s|z-1|^2=(1\pm s)\abs{z\mp \frac{s}{1\pm s}}^2-\frac{1}{1\pm s}$, $\mu(z)\leq 0$ on the set bounded by the two circles $\abs{z-\frac{s}{1+s}}=\frac{1}{1+s}$ and $\abs{z+\frac{s}{1-s}}=\frac{1}{1-s}$, as in Example 1 in Figure \ref{fig:examples}.  Observe that when $\mu(z)=0$ and $z$ is bounded away from one, then $\mu$ is a defining function for a circle, and hence hypothesis (4) of Lemma \ref{lem:worm_like_domain} is satisfied. It remains to see that this hypothesis is satisfied on some neighborhood of $z=1$.

We introduce the non-isotropic distance
\[
  d(z)=(\re z-1)^2+(\im z)^4.
\]
When $z$ is close to $1$, we have $|\mu(z)|\leq O(d(z))$.  Extracting the terms of order $O(d(z))$, we have
\[
  \mu(z)=4(\re z-1)^2+4(\re z-1)(\im z)^2+(1-s^2)(\im z)^4+O\left((d(z))^{3/2}\right).
\]
Fix $\frac{1-s}{2}<a<\frac{5-s^2}{2(5+s^2)}$ and set
\[
  b=\frac{4-(3+s^2)a}{3+s^2+4a}=\frac{(1-2a)^2-s^2}{3+s^2+4a}+1-a=\frac{5-s^2-2(5+s^2)a}{6+2s^2+8a}+\frac{1}{2}.
\]
Then $\frac{1-s}{2}<a<\frac{1}{2}<b<\frac{1+s}{2}$, $a+b<1$, and $4ab+(a+b)(3+s^2)-4=0$.  If $z$ satisfies $-a(\im z)^2\geq\re z-1\geq-b(\im z)^2$, then
\begin{multline*}
  0\leq(\re z-1+b(\im z)^2)(-(\re z-1)-a(\im z)^2)\\
  =(a+b-1)d(z)-\frac{a+b}{4}(4(\re z-1)^2+4(\re z-1)(\im z)^2+(1-s^2)(\im z)^4),
\end{multline*}
so
\[
  \mu(z)\leq-\frac{4(1-a-b)}{a+b}d(z)+O\left((d(z))^{3/2}\right).
\]
Hence, for $d(z)$ sufficiently small, we will have $\mu(z)\leq 0$ when $-a(\im z)^2\geq\re z-1\geq-b(\im z)^2$.  Equivalently, there exists $r_{a,b}>0$ such that if $d(z)<r_{a,b}$ and $\mu(z)>0$, then $-a(\im z)^2<\re z-1$ or $\re z-1<-b(\im z)^2$.  If we differentiate polynomial terms with respect to $\re z$, then the exponent of $d(z)$ in the upper bound will decrease by at most $\frac{1}{2}$, while differentiating by $\im z$ will decrease this exponent by at most $\frac{1}{4}$.  Hence, differentiating in $z$ will decrease the exponent by at most $\frac{1}{2}$.  In particular,
\[
  \frac{\partial\mu}{\partial z}=4(\re z-1)+2(\im z)^2+O\left((d(z))^{3/4}\right).
\]
Fix $z$ satisfying $d(z)<r_{a,b}$ and $\mu(z)>0$.  If $\re z-1\geq 0$, then $4(\re z-1)+2(\im z)^2\geq 2\sqrt{d(z)}$.  If $-a(\im z)^2<\re z-1<0$, then $4(\re z-1)+2(\im z)^2>\frac{2-4a}{1+a}\left((1-\re z)+(\im z)^2\right)\geq\frac{2-4a}{1+a}\sqrt{d(z)}$.  If $\re z-1<-b(\im z)^2$, then $4(\re z-1)+2(\im z)^2<-\frac{4b-2}{1+b}\left((1-\re z)+(\im z)^2\right)<-\frac{4b-2}{1+b}\sqrt{d(z)}$.  In any case, $\abs{4(\re z-1)+2(\im z)^2}> \min\{\frac{2-4a}{1+a},\frac{4b-2}{1+b}\}\sqrt{d(z)}$.  Hence, whenever $d(z)<r_{a,b}$ and $\mu(z)>0$ we have
\[
  \abs{\frac{\partial\mu}{\partial z}}\geq \min\set{\frac{2-4a}{1+a},\frac{4b-2}{1+b}}\sqrt{d(z)}-O\left((d(z))^{3/4}\right).
\]
Since $\abs{\frac{\partial^2\mu}{\partial z\partial\bar z}}\leq O\left(1\right)$, we have enough information to confirm that hypothesis (4) of Lemma \ref{lem:worm_like_domain} is satisfied near $z=1$ for any $m>1$.

Example 1 from Figure \ref{fig:examples} provides an example for Proposition \ref{prop:annulus} for any $\gamma>\frac{1}{2}$, but now we wish to construct a family of examples for which the hypotheses of Proposition \ref{prop:annulus} are satisfied but $\gamma$ may need to be arbitrarily close to $1$.  We fix integers $j>3$ and $j>k\geq 2j/3$ along with a constant $0<s<1$ and set
\[
  \mu(z)=(|z|^2-1)^{2j}-|z-1|^{4(j-k)}(2\im z)^{2k}-s^2|z-1|^{4j}.
\]
One can check that $\mu$ will also have the necessary properties when $2j/3>k>0$, but the error terms will require additional cases, so we restrict to the case in which $k>2j/3$, since this will suffice to take $\gamma$ arbitrarily close to $1$.  We have $\mu(0)=1-s^2>0$ and
\[
  \mu(e^{i\theta})=-(2-2\cos\theta)^{2(j-k)}(2\sin\theta)^{2k}-s^2(2-2\cos\theta)^{2j},
\]
which is strictly negative unless $e^{i\theta}=1$.  When $z\neq 1$, we introduce the biholomorphic change of coordinates $z=\frac{iw-1}{iw+1}$ and compute
\[
  \mu(z(w))=\abs{\frac{2}{iw+1}}^{4j}\left((\im w)^{2j}-(\re w)^{2k}-s^2\right).
\]
Near $w=i$, $\mu(z(w))$ has principal part $(1-s^2)\abs{\frac{2}{iw+1}}^{4j}$, and hence $\mu(z)$ is uniformly bounded away from zero as $|z|\rightarrow\infty$.  Therefore, $\mu(z)$ satisfies hypotheses (1)-(3) of Lemma \ref{lem:worm_like_domain}.  Furthermore, $\mu(z(w))=0$ only along the pair of smooth curves parameterized by $\im w=\pm((\re w)^{2k}+s^2)^{1/2j}$.  We easily check that $\frac{\partial}{\partial w}\mu(z(w))\neq 0$ when $\mu(z(w))=0$, so $\frac{\partial\mu}{\partial z}\neq 0$ whenever $\mu(z)=0$ and $z\neq 1$.  Hence, if our requirements on $\mu$ will fail, they can only fail in a neighborhood of $z=1$.

To study the behavior near $z=1$, we introduce the nonisotropic distance
\[
  d(z)=(\re z-1)^{2j}+(\im z)^{4j-2k}.
\]
On some neighborhood of $z=1$, $|\mu(z)|\leq O(d(z))$.  To get a more precise approximation for $\mu(z)$ when $z$ is close to $1$, we first note that since $4j-2k>2j$ we have
\[
  |z-1|^{4j}\leq O\left((d(z))^{1+k/(2j-k)}\right).
\]
Also near $z=1$, we have
\[
  \abs{|z-1|^2-(\im z)^2}\leq O\left((d(z))^{1/j}\right),
\]
so, using $\frac{1}{j}+\frac{4j-2k-2}{4j-2k}=1+\frac{j-k}{j(2j-k)}$, we obtain
\[
  \abs{|z-1|^{4(j-k)}(2\im z)^{2k}-2^{2k}(\im z)^{4j-2k}}\leq O\left((d(z))^{1+(j-k)/(j(2j-k))}\right).
\]
Since $|z|^2-1=|z-1|^2+2(\re z-1)$, we have
\[
  \abs{|z|^2-1-2(\re z-1)}\leq O\left((d(z))^{1/(2j-k)}\right),
\]
so using $\frac{1}{2j}<\frac{1}{2j-k}<\frac{1}{j}$ and $\frac{1}{2j-k}+\frac{2j-1}{2j}=1+k/(2j(2j-k))$, we obtain
\[
  \abs{(|z|^2-1)^{2j}-2^{2j}(\re z-1)^{2j}}\leq O\left((d(z))^{1+k/(2j(2j-k))}\right).
\]
Comparing our error terms, we see that since $k>\frac{2j}{3}$, $O\left((d(z))^{1+(j-k)/(j(2j-k))}\right)$ is the dominant error term.  Hence,
\begin{equation}
\label{eq:mu_approximation}
  \mu(z)=2^{2j}(\re z-1)^{2j}-2^{2k}(\im z)^{4j-2k}+O\left((d(z))^{1+(j-k)/(j(2j-k))}\right).
\end{equation}
Note that if $|\im z|<|\re z-1|^{j/(2j-k)}$, then $2^{2j}(\re z-1)^{2j}-2^{2k}(\im z)^{4j-2k}>(2^{2j-1}-2^{2k-1})d(z)$.  Hence, there exists $r>0$ such that if $|\im z|<|\re z-1|^{j/(2j-k)}$ and $d(z)<r$, then $\mu(z)>0$.  Therefore, the hypotheses of Proposition \ref{prop:annulus} are satisfied for $\gamma=\frac{j}{2j-k}$.

Fix $0<c<2^{-(j-k)/j}$.  If $|\re z-1|\leq c|\im z|^{(2j-k)/j}$, then $2^{2j}(\re z-1)^{2j}-2^{2k}(\im z)^{4j-2k}\leq-\frac{2^{2k}-2^{2j}c^{2j}}{1+c^{2j}}d(z)$.  Once again, there must exist $r_c>0$ such that whenever $\mu(z)>0$ and $d(z)< r_c$, then we have $|\re z-1|>c|\im z|^{(2j-k)/j}$.  Differentiating a polynomial bounded by some power of $d(z)$ with respect to $\re z$ will decrease the exponent of this bound by at most $\frac{1}{2j}$, and differentiating with respect to $\im z$ will decrease the exponent by at most $\frac{1}{4j-2k}$.  Hence, differentiating in $z$ will decrease the exponent of the upper bound by at most $\frac{1}{2j}$.  Since the derivative of $(\im z)^{4j-2k}$ with respect to $z$ is of order $O\left((d(z))^{1-1/(4j-2k)}\right)$, the derivatives of the polynomials in the error term in \eqref{eq:mu_approximation} with respect to $z$ are of order $O\left((d(z))^{1-k/(2j(2j-k))}\right)$, and $\frac{1}{4j-2k}>\frac{k}{2j(2j-k)}$, we have
\[
  \frac{\partial\mu}{\partial z}=j 2^{2j}(\re z-1)^{2j-1}+O\left((d(z))^{1-1/(4j-2k)}\right).
\]
Suppose $\mu(z)>0$ and $d(z)<r_c$.  Then $|\re z-1|>c|\im z|^{(2j-k)/j}$, so $(\re z-1)^{2j}>\frac{c^{2j}}{1+c^{2j}}d(z)$.  Hence,
\[
  \abs{\frac{\partial\mu}{\partial z}}>j2^{2j}\left(\frac{c^{2j}}{1+c^{2j}}d(z)\right)^{1-1/(2j)}-O\left((d(z))^{1-1/(4j-2k)}\right).
\]
Since $\abs{\frac{\partial^2\mu}{\partial z\partial\bar z}}\leq O\left((d(z))^{1-1/j}\right)$, we can check that the hypotheses of Proposition \ref{lem:worm_like_domain} are satisfied for any $m\geq 1$.  Furthermore, we have a concrete family of examples for which Proposition \ref{prop:annulus} and Corollary \ref{cor:annulus} hold for $\gamma$ arbitrarily close to $1$.

Unfortunately, this leaves open the case where the set of infinite type points is topologically equivalent to the set considered in Proposition \ref{prop:annulus}, but satisfies an interior cone condition, as in Example 2 in Figure \ref{fig:examples}.  The methods of this section can easily be extended to construct an example generated by $\mu(z)=|z-1|^{4(k-j)}(|z|^2-1)^{2j}-(2\im z)^{2k}-s|z-1|^{4k}$ where $k\geq j>0$, so such examples exist, but it is not clear if good vector fields exist or what the Diederich-Fornaess Index will equal.  We believe that understanding these open questions may shed greater light on the relationship between families of good vector fields and the Diederich-Fornaess Index.

\bibliographystyle{amsplain}
\bibliography{harrington}
\end{document}